\documentclass[preprint]{imsart}
\usepackage{amsmath, latexsym, amsthm, amsfonts,bm,amssymb} 
\usepackage[dvipsnames]{xcolor}
\usepackage{appendix,todonotes}

\usepackage{natbib} 
\usepackage{url} 
\usepackage{enumerate}
\usepackage{mathtools} 

\usepackage{natbib} 
\bibpunct{(}{)}{;}{a}{}{,} 

\theoremstyle{plain}
\newtheorem{thm}{Theorem}
\newtheorem{prop}[thm]{Proposition}
\newtheorem{lemma}[thm]{Lemma}
\newtheorem{cor}[thm]{Corollary}

\usepackage[normalem]{ulem}

\theoremstyle{remark}
\newtheorem{rem}[thm]{Remark}
\newtheorem{exam}[thm]{Example}

\newcommand{\bcdot}{\boldsymbol{\cdot}}

\def\var{{\rm {\mathbb Var\,}}}
\def\pvar{{\rm {\mathbb Var}_{\Pi_n}}}

\newcommand{\eps}{\varepsilon}
\newcommand{\half}{\frac{1}{2}}

\newcommand{\pp}{\mathbb{P}}

\newcommand{\ee}{\mathbb{E}}

\newcommand{\cf}{\mathcal{F}}

\newcommand{\dd}{\mathrm{d}}
\newcommand{\Pp}{\mathbb{P}}

\newcommand{\fpspace}{(\Omega,\mathcal{F},\mathbb{F},\Pp)}

\newcommand{\ig}{\operatorname{IG}}

\newcommand{\pot}{\mathrm{POT}}
\newcommand{\one}{\mathbf{1}}

\allowdisplaybreaks

\usepackage{tikz,hyperref}

\definecolor{lime}{HTML}{A6CE39}
\DeclareRobustCommand{\orcidicon}{%
	\begin{tikzpicture}
	\draw[lime, fill=lime] (0,0) 
	circle [radius=0.16] 
	node[white] {{\fontfamily{qag}\selectfont \tiny ID}};
	\draw[white, fill=white] (-0.0625,0.095) 
	circle [radius=0.007];
	\end{tikzpicture}
	\hspace{-2mm}
}

\foreach \x in {A, ..., Z}{%
	\expandafter\xdef\csname orcid\x\endcsname{\noexpand\href{https://orcid.org/\csname orcidauthor\x\endcsname}{\noexpand\orcidicon}}
}

\begin{document}

\begin{frontmatter}
\title{Nonparametric Bayesian volatility estimation for gamma-driven stochastic differential equations}
\runtitle{Bayesian volatility estimation}
\runauthor{Belomestny, Gugushvili, Schauer, Spreij}
\begin{aug}
\author[A]{\fnms{Denis} \snm{Belomestny}\ead[label=e1]{denis.belomestny@uni-due.de}},
\author[B]{\fnms{Shota} \snm{Gugushvili}\ead[label=e2]{shota@yesdatasolutions.com}}
\and
\author[C]{\fnms{Moritz} \snm{Schauer}\ead[label=e3]{smoritz@chalmers.se}}
\and
\author[D]{\fnms{Peter} \snm{Spreij}\ead[label=e4]{spreij@uva.nl}}
\address[A]{Department,
Faculty of Mathematics\\
Duisburg-Essen University\\
Thea-Leymann-Str.~9\\
D-45127 Essen\\
Germany\\
and Faculty of Computer Sciences\\
HSE University\\
Moscow, Russian Federation
\printead{e1}}

\address[B]{Biometris\\
	Wageningen University \& Research\\
	Postbus 16\\
	6700 AA Wageningen\\
	The Netherlands
\printead{e2}}

\address[C]{Department of Mathematical Sciences\\
	Chalmers University of Technology and University of Gothenburg\\
	SE-412 96 G\"{o}teborg\\
	Sweden
\printead{e3}}

\address[D]{Korteweg-de Vries Institute for Mathematics\\
Universiteit van Amsterdam\\
P.O.\ Box 94248\\
1090 GE Amsterdam\\
The Netherlands\\
and Institute for Mathematics, Astrophysics and Particle Physics\\
Radboud University\\
Nijmegen\\
Netherlands
\printead{e4}}

\end{aug}

\begin{abstract}
We study a nonparametric Bayesian approach to estimation of the volatility function of a stochastic differential equation driven by a gamma process. The volatility function is modelled a priori as piecewise constant, and we specify a gamma prior on its values. This leads to a straightforward procedure for posterior inference via an MCMC procedure.  We give theoretical performance guarantees (minimax optimal contraction rates for the posterior) for the Bayesian estimate in terms of the regularity of the unknown volatility function. We  illustrate the method on synthetic  and real data examples.
\end{abstract}

\begin{keyword}
\kwd{Gamma process}
\kwd{Nonparametric Bayesian estimation}
\kwd{Stochastic differential equation}
\end{keyword}

\end{frontmatter}

\section{Introduction}
\subsection{Problem formulation}

The goal of the present paper is Bayesian nonparametric estimation of the positive local scale function $\sigma$ that appears in the L\'evy-driven stochastic differential equation 
\begin{equation} \label{eq:sde}
\dd X_t=\sigma(X_{t-})\,\dd L_t,\, X_0= 0, 
\end{equation}
from observations of the solution $X$. Here
$L$ is a gamma process with $L_0=0$ and therefore $L$ is a subordinator, i.e.\ a stochastic process with monotonous sample paths. Furthermore, $L$ has
a L\'evy measure $\nu$ admitting the L\'evy density
\begin{equation}\label{eq:v}
v(x)= \alpha x^{-1}\exp(-\beta x),\, x>0,
\end{equation}
where $\alpha$ and $\beta$ are two positive constants. The process $L$ has independent increments,  and $L_t-L_s$ has a Gamma$(\alpha (t-s),\beta)$ distribution for $t>s$ with shape parameter $\alpha$ and scale parameter $1/\beta$, as defined in Section~\ref{subsec:notation}.

Under the assumption that the function $\sigma$ (in view of financial applications we refer to it as volatility function) is measurable and satisfies a linear growth condition, it has been shown in \cite{bgssweak} that Equation~\eqref{eq:sde} admits a weak solution that is unique in law. Under the stronger condition that $\sigma$ is 
Lipschitz continuous, Equation~\eqref{eq:sde} even has a unique strong solution, see~\cite[Theorem~V.6]{Protter2004}. Note that $X$ is a Markov process.

\begin{exam}\label{exam:1}
To get an impression of how observations of $X$ look like, we take $X$ to solve the L\'evy SDE \eqref{eq:sde} with the volatility function
\begin{equation}
\label{vol_ex1}
\sigma_0(x) = \frac{1}{500}\left(\frac{3}{2}+\sin(2\pi x)\right).
\end{equation}
For the driving gamma process $L$ we take parameters $\alpha =1$, $\beta = 1$.
We used the Euler scheme to generate a single trajectory to serve as observation. 
We simulated a path traversing the unit interval on a time grid with time step-size $\delta t = 0.0001$ starting in $X_0 = 0 = b_0$. In Section~\ref{subsec:prior} we will introduce bins with boundaries $b_k$ and their hitting times $\tau_k$.
\emph{Approximate} hitting times were obtained setting $\tau'_k = \inf\{i \delta t   \colon X_{i \delta t}  \ge b_k\}$.  Figure~\ref{fig:ex1obs} shows the sample path until the hitting time $\tau'_K$ of $b_K = 1$.

\begin{figure}
\begin{center}
\includegraphics[width=0.6\linewidth]{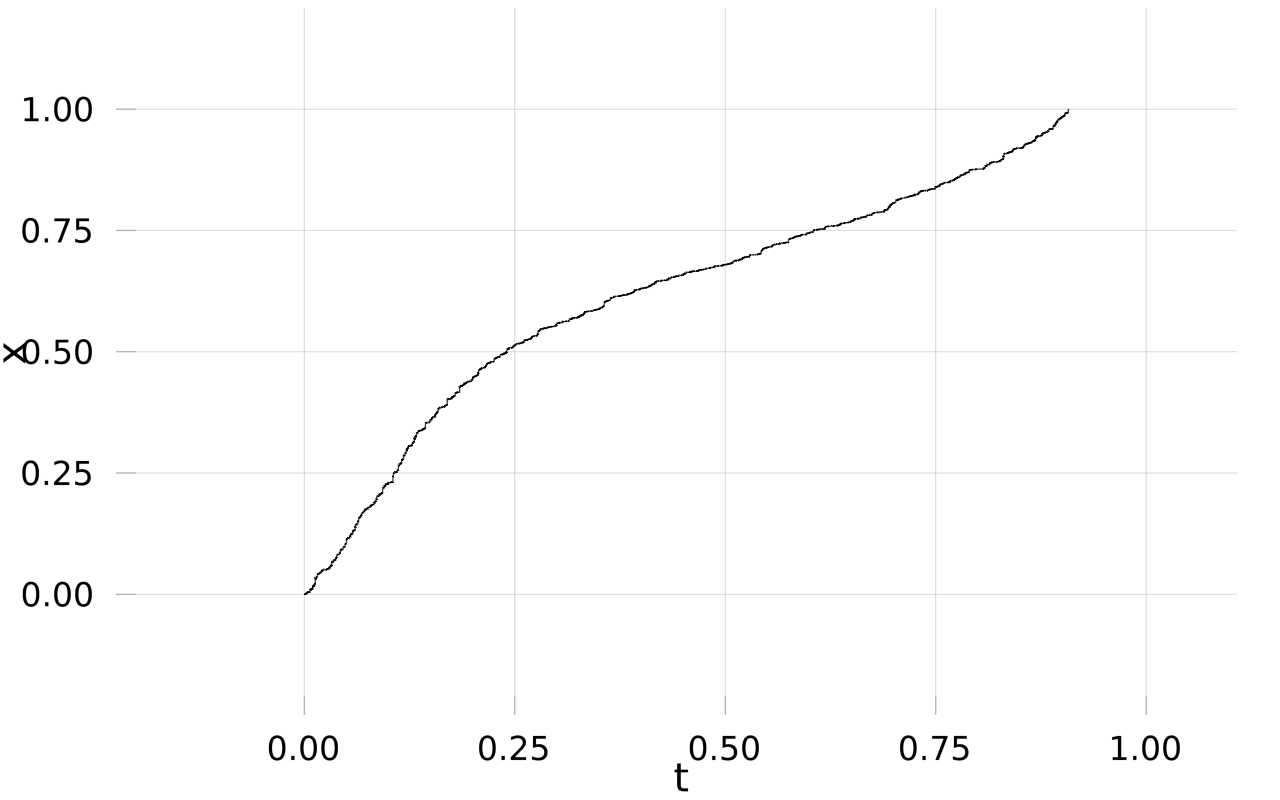}
\end{center}
\caption{Observation $X^n$, $n = 200$, generated from \eqref{eq:sde} with the Euler scheme with volatility $\sigma_0$ given in \eqref{vol_ex1}. Parameters of the driving gamma process are $\alpha =1$, $\beta = 1$.}
\label{fig:ex1obs}
\end{figure}

\end{exam}
\subsection{Motivation}
Gamma processes, that form a special class of L\'evy processes (see, e.g., \cite{kyprianou14}), are a fundamental modelling tool in several fields, e.g. reliability (see \cite{noortwijk09}) and risk theory (see \cite{dufresne91}). Since the driving gamma process $L$ in \eqref{eq:sde} has non-decreasing sample paths and the volatility function $\sigma$ is non-negative, also the process $X$ has non-decreasing sample paths. Such processes find applications across various fields. A reliability model as in \eqref{eq:sde} has been thoroughly investigated from a probabilistic point of view in \cite{wenocour89}, and constitutes a far-reaching generalisation of a basic gamma model. Furthermore, non-decreasing processes are ideally suited to model revenues from an innovation: in \cite{chance08}, the authors study the question of pricing options on movie box office revenues that are modelled through a gamma-like stochastic process. Another potential application is in modelling the evolution of forest fire sizes over time, as in \cite{reed02}.

Any practical application of the model \eqref{eq:sde} would require knowledge of the volatility function $\sigma$, that has to be inferred from observations on the process $X$. In this paper, we will approach estimation of $\sigma$ nonparametrically. The latter comes in handy when no apparent functional form for the volatility is available, which usually is the case in practice. Such an approach reduces the risk of model misspecification and allows a honest representation of uncertainties in inferential conclusions (see \cite{silverman86} and \cite{mueller13}). Recent years have witnessed a tremendous growth of interest and rapid advances in nonparametric Bayesian methods (see, e.g., two monographs \cite{ghosal17} and \cite{mueller15}). Coherence, elegance and automatic uncertainty quantification are some of the widely acknowledged attractive features of a Bayesian approach to statistics. Hence our decision to follow a Bayesian method in this paper. However, we also note that with a careful choice of a prior a nonparametric Bayesian method enjoys very favourable frequentist properties; see, e.g., \cite{ghosal17}.

On the theoretical side,
we are ultimately interested in asymptotic properties of our Bayesian procedure for estimating the volatility function $\sigma$. To that end we need a sufficiently rich set of observations and this can be accomplished by the scaled observation process in \eqref{eq:sden} below, instead of $X$ satisfying \eqref{eq:sde}. So we consider the process $X^n$ given as the solution to
\begin{equation}\label{eq:sden}
\dd X^n_t=\frac{1}{n}\sigma(X^n_{t-})\,\dd L_t,\, X^n_0=0.  
\end{equation}
The scaling factor $\frac{1}{n}$ causes for large values of $n$ a `slow growth' of the process $X^n$ and `long times' to reach certain levels.
We will thus assume that $X^n$ is observed on a long time interval $[0,T^n]$, where $T^n\to\infty$. Later we will specify $T^n$ and we will see that $T^n$ grows roughly proportionally with $n$. Asymptotic results will be derived for $n\to\infty$. The setup above for having a rich set of observations allows for a different, but equivalent description of the model. Let $Y^n_t=X^n_{nt}$. Then $Y^n$ satisfies the SDE
\begin{equation}\label{eq:sdey}
Y^n_{t}=\int_0^{t}\sigma(Y^n_{s-})\,\dd L^n_{s},
\end{equation}
where $L^n_t=\frac{L_{nt}}{n}$. Note that $L^n$ is again a gamma process with a L\'evy density  
\[
v_n(x)=\frac{n\alpha}{x}\exp(-n\beta x),
\]
and (for $t>s\geq 0$) $L^n_t-L^n_s$ has a Gamma$(n\alpha (t-s), n\beta)$ distribution. Here we see  a smoother behaviour of $L^n$ and hence of $Y^n$ for growing $n$. In fact, it can be shown that $L^n$ weakly converges to the function $L^\infty$ given by $L^\infty_t=\frac{\alpha}{\beta}t$, and as a consequence $Y^n$ should then converge to a deterministic limit as well. This behaviour of the process $Y^n$ bears some similarity to the properties of a diffusion process $Y^\eps$  with small diffusion coefficient ($W$ is a Wiener process), 
\[
Y^\eps_t=\int_0^t a(Y^\eps_s)\,\dd s + \eps W_t,
\]
which also has a deterministic limit as $\eps\to 0$. The similarity becomes more pronounced, if one writes the semimartingale decomposition of $Y^n$,
\[
Y^n_{t}=\frac{\alpha}{\beta}\int_0^{t}\sigma(Y^n_{s-})\,\dd s+M^n_t,
\]
where, under appropriate conditions, the (local) martingale $M^n$ vanishes for $n\to\infty$.
This conceptual similarity with drift estimation in a well-known diffusion model and the statistical problem there, estimation of the function $a$, should be understood as a mathematical motivation for our study.
Although it is possible to present all results that follow in terms of properties of the process $Y^n$, we opted to give them for $X^n$.

\subsection{Literature overview}
Statistical inference for L\'evy driven SDEs is an active research with many contributions. 
Let us mention recent works by \cite{jasra2011inference}, \cite{jasra2019bayesian},\cite{uehara2019statistical}, \cite{gushchin2019drift} and \cite{eguchi2020schwartz}. Such models are popular in finance and econometrics, see e.g. \cite{todorov2011econometric}. In this paper we study the problem of Bayesian inference for the volatility coefficient of a L\'evy-driven SDE. 

Nonparametric Bayesian literature on inference in the model \eqref{eq:sde} is non-existent, but somewhat related problems have been considered in several papers. Thus, \cite{belomestny18} study a nonparametric Bayesian approach to estimation of the L\'evy measure for L\'evy processes with monotonous sample paths (subordinators), while \cite{gugu15}, \cite{gugu16} and \cite{gugu19} investigate the same problem for compound Poisson processes. There exists also a limited amount of work on nonparametric Bayesian volatility estimation in diffusion models, e.g.\ \cite{batz17} and \cite{nickl17}, but a paper that is most related to the present one is \cite{gugu18b}. Finally, \cite{koskela17} is a theoretical contribution, where frequentist consistency of a Bayesian approach to inference in jump-diffusion models is established.

\subsection{Our contribution}

Our work is the first contribution to nonparametric Bayesian volatility estimation for L\'evy-driven SDEs. The method we propose is easy to understand, and leads to good practical results in synthetic and real data examples. We expect it to open up new research directions in inference for stochastic differential equations driven by jump processes, both from the practical and theoretical point of views. Our approach is based on a \emph{piecewise constant  approximation} with a proper prior on the corresponding parameters. We show the contraction of the posterior for H\"older continuous volatility coefficients and propose a MCMC  procedure for sampling. 

\begin{exam}
We continue Example~\ref{exam:1} and graphically illustrate our inferential results. In agreement with the notation pertaining to the asymptotic regime detailed in \eqref{eq:sden}
we took $X^n$ as the solution to the L\'evy SDE  \eqref{eq:sde} with the volatility function
\begin{eqnarray}
\label{eq:sigma0}
\sigma_0(x) = \frac{3}{2}+\sin(2\pi x)
\end{eqnarray}
and scaled it with $1/n$, $n=500$. For details of our Bayesian procedure we refer to Section~\ref{subsec:prior}, in the present example this comes down to the following. 

We have partitioned the unit interval into $10$ bins setting $b_k = k \delta x$, $\delta x = 0.1,$ $k=1,\ldots,10,$ and used the piecewise constant prior of the form $\sigma(x)=\sum_k \xi_k \mathbf{1}_{B_k}(x),$ where $\xi_1,\ldots,\xi_{10}$ are i.i.d.\ random variables with an inverse gamma distribution. The posterior is given in closed form  by \eqref{eq:igposterior}. Figure~\ref{fig:ex1post} contrasts the corresponding marginal  $90\,\%$-posterior credible bands for $\sigma$ with the true volatility function $\sigma_0$ from 
\eqref{eq:sigma0}.
\begin{figure}
\begin{center}
\includegraphics[width=0.7\linewidth]{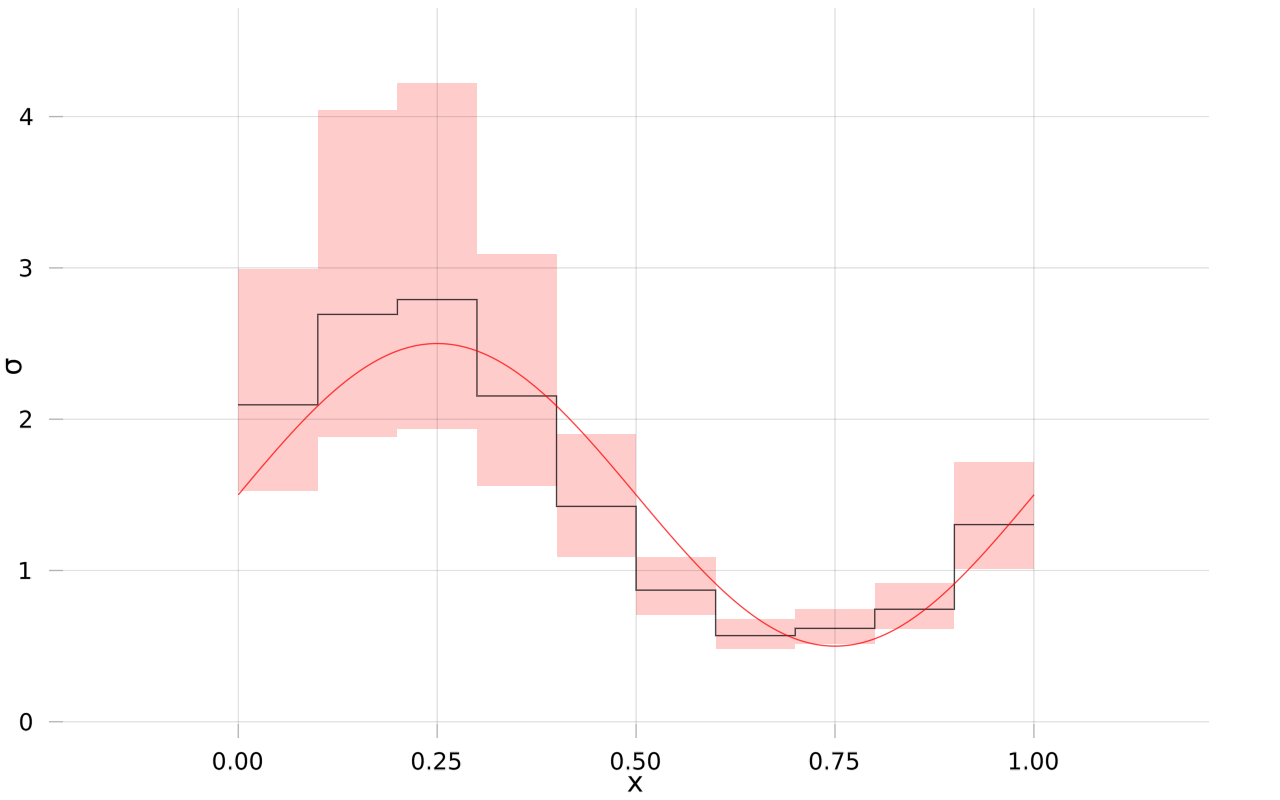}
\end{center}
\caption{Red: true volatility $\sigma_0$ as a function of $x$. Shaded: marginal $90\,\%$-posterior credible band for the piecewise constant posterior $\sigma$. Black: marginal posterior median. }
\label{fig:ex1post}
\end{figure}
\end{exam}

\subsection{Structure of the paper}

The paper is organised as follows: in Section~\ref{sec:approach} we describe in detail our Bayesian method for volatility estimation. 
In Section~\ref{section:pw} we present asymptotic properties of the posterior distribution when the volatility function $\sigma$ is piecewise constant, whereas in Section~\ref{section:holder} we present such properties when the volatility function is H\"older continuous. 
The contraction rate we obtain in the latter case is shown to be minimax optimal.
Real data examples are considered in 
Section~\ref{sec:real}. In Section~\ref{sec:extensions} we consider some extensions and variations of our approach, in particular paying attention to the setting of discrete time observations, and propose a Metropolis-Hastings approach for simulating gamma process bridges.

\subsection{Notation}\label{subsec:notation}
We denote the gamma distribution with shape parameter $a>0$ and rate parameter $b>0$  (hence scale $1/b$) by $\operatorname{Gamma}(a,b)$. Recall that its density is given by
\[
x \mapsto \frac{b^a}{\Gamma(a)} x^{a-1}e^{-b x}, \quad x>0,
\]
where $\Gamma$ is the gamma function. The inverse gamma distribution with shape parameter $a>0$ and scale parameter $b>0$ will be denoted by $\ig(a,b)$. The corresponding density is
\[
x \mapsto \frac{b^a}{\Gamma(a)} x^{-a-1} e^{-b/x}, \quad x>0,
\]
and its expectation and variance are $\frac{b}{a-1}$ and $\frac {b ^{2}}{(a-1)^{2}(a-2)},$ respectively. Following a standard Bayesian convention, we will often use lowercase letters to write random variables. Conditioning of $x$ on $y$ will be denoted by $x\mid y$.

\section{Bayesian approach}
\label{sec:approach}

To compute a posterior distribution,  a likelihood ratio is needed. In this section we study likelihood ratios and existence of a weak solution to Equation~\eqref{eq:sde}. 

\subsection{Likelihood}

Let $\fpspace$ be a filtered probability space and let $(L_t)_{t\geq 0}$ be a gamma process adapted to $\mathbb{F}$, whose L\'evy measure admits the density $v$ given by \eqref{eq:v}.
Assume that $X$ is a (weak) solution to \eqref{eq:sde},
and assume that $X$ is observed on an interval $[0,T]$. We denote by $\pp^\sigma_T$ (a probability measure on $\cf^X_T=\sigma(X_t, 0\leq t\leq T)$) its law. In agreement with this notation we let $\pp^1_T$ be the law of $X$ when $\sigma\equiv1$, in which case $X_t=L_t,\, t\in [0,T]$. The measure $\pp^1_T$ will serve as a reference measure. The choice $\sigma=1$ for obtaining a reference measure is natural, but also arbitrary. Many other choices for the function $\sigma$ are conceivable, in particular other constant functions. In \cite{bgssweak}  the following proposition has been proven.

\begin{prop}\label{prop:z}
Let $\sigma$ be a positive locally bounded measurable function on $[0,
\infty)$ such that \eqref{eq:sde} admits a weak solution that is unique in law. Assume moreover that $\sigma$ is lower bounded by a strictly positive constant. Let $T$ be a  finite (stopping)  time $T>0$. Then the laws $\pp^\sigma_T$ and $\pp^1_T$ are equivalent on $\cf^X_T$ and the corresponding Radon-Nikodym derivative $Z_T$  
has the explicit representation: 
\begin{equation}\label{eq:explicit}
Z_T=\exp\left(\int_{0}^{T}\int_{0}^{\infty}\log Y(t,x)\,\mu^X(\dd x,\dd t)-  \int_0^T\int_{0}^{\infty}(Y(t,x)-1)v(x)\,\dd x\,\dd t  \right),
\end{equation}
where both double integrals are a.s.\ finite. 
\end{prop}

\begin{rem}\label{remark:explicit}
For statistical inference on $\sigma$ one needs a realisation of the random quantity $Z_T$, induced by an observed realisation of $X$. The realisation of $Z_T$ is then simply obtained by evaluation of the integrals along a path of $X$. This causes no difficulties as the integrals in $Z_T$ are defined pathwise.
We will need (in Corollary~\ref{cor:lik} below) the result for positive piecewise constant functions $\sigma$, and we will see there that for a realisation of $Z_T$ one only needs the observed values of the $X^n_{\tau^n_k}$. The stopping times $\tau^n_k$ will be specified later.
\end{rem}

\subsection{Piecewise constant volatility prior}\label{subsec:prior}

In our non-parametric Bayesian approach, we will a priori model $\sigma$ as a piecewise constant function. Namely,
\begin{eqnarray}
\label{eq:sigma_pc}
\sigma(x)=\sum_{k=1}^K \xi_k \mathbf{1}_{B_k}(x)
\end{eqnarray}
for bins $B_1=[0,b_1]$,
$
B_k=(b_{k-1},b_k], k=2,\ldots,K-1,
$
and  $B_K=(b_{K-1},b_K)$, with appropriately chosen increasing sequence of bin endpoints $\{b_k\}$ and the bin number $K$. The bins $B_k$'s form a partition of the positive halfline $[0,\infty)$. 
The $\{\xi_k\}$ are positive numbers (later on positive random variables). Although we use \eqref{eq:sigma_pc} for our model, we emphasize that the `true' $\sigma$ does not need to be piecewise constant. As a final remark we note that when $\sigma$ is given by \eqref{eq:sigma_pc}, \eqref{eq:sden} still has a unique solution, obtained as concatenation of stopped gamma processes.

\begin{cor}\label{cor:lik}
Suppose that $X^n$ is given by \eqref{eq:sden} with \(\sigma\) given by \eqref{eq:sigma_pc}. Let \(\tau^n_{k}=\inf\{t\geq 0: \, X^n_t\geq b_k\}\), \(k=1,\ldots,K,\)
and write $T=\tau^n_K$. Then
\begin{eqnarray*}
\frac{\dd\pp^\sigma_T}{\dd\pp^1_T}=\exp\left\{ \beta\sum_{k=1}^{K}\left[1-n\xi_{k}^{-1}\right](X^n_{\tau^n_{k}}- X^n_{\tau^n_{k-1}})-\alpha\sum_{k=1}^{K}\left(\tau^n_{k}-\tau^n_{k-1}\right)\log(\xi_{k}/n)\right\}. 
\end{eqnarray*}
\end{cor}
\begin{proof}
We have now $Y(t,x)=\exp(-\beta x(\frac{n}{\sigma(X^n_{t-})}-1))$. It follows from Proposition~\ref{prop:z} and Remark~\ref{remark:explicit} that 
\begin{eqnarray*}
\frac{\dd\pp^\sigma_T}{\dd\pp^1_T} 
 & = & \exp\left\{ \int_{(0,T]}\int_{(0,\infty)}\left\{ x\beta\left[1-\frac{n}{\sigma(X^n_{s-})}\right]\right\} \mu^X(\dd s,\dd x) \right.
 \\
 && \left. -\int_{0}^{T}\int_{(0,\infty)}\left(e^{x\beta\left[1-\frac{n}{\sigma(X^n_{s-})}\right]}-1\right)v(x)\,\dd x\,\dd s\right\}=\exp(I_1-I_2), 
\end{eqnarray*}
where
\begin{align*}
I_1 & =  \int_{(0,T]}\int_{(0,\infty)}\left\{ x\beta\left[1-\frac{n}{\sigma(X^n_{s-})}\right]\right\} \mu^X(\dd s,\dd x) \\
& = \sum_{k=1}^K \int_{(\tau^n_{k-1},\tau^n_k]}\int_{(0,\infty)}\left\{ x\beta\left[1-\frac{n}{\sigma(X^n_{s-})}\right]\right\} \mu^X(\dd s,\dd x) \\
& =\sum_{k=1}^K\beta\left[1-\frac{n}{\xi_{k}}\right]\int_{(\tau^n_{k-1},\tau^n_k]}\int_{(0,\infty)}x \mu^X(\dd s,\dd x) =\beta\sum_{k=1}^K\left[1-\frac{n}{\xi_{k}}\right](X^n_{\tau^n_{k}}- X^n_{\tau^n_{k-1}}),
\end{align*}
since $\sigma(X^n_{s-})=\xi_k$ for $s\in(\tau^n_{k-1},\tau^n_k]$ if $\tau^n_{k-1}<\tau^n_k$. But the above expression is also valid if $\tau^n_{k-1}=\tau^n_k$, since then $X^n_{\tau^n_{k}}- X^n_{\tau^n_{k-1}}=0$.
Furthermore, by similar reasoning, 
\begin{align*}
I_2=\sum_{k=1}^K\left(\tau^n_{k}-\tau^n_{k-1}\right)\int_{\mathbb{R}_{+}}\frac{\alpha}{x}\left[e^{-x\beta n\xi_{k}^{-1}}-e^{-\beta x}\right]\,\dd x.
\end{align*}
We need an intermediate result. For any $c>0$ one has 
$\int_{\mathbb{R}_{+}}\frac{1}{x}\left[e^{-x\beta/c}-e^{- x\beta}\right]\,\dd x=\int_{\mathbb{R}_{+}}\frac{1}{x}\left[e^{-x/c}-e^{- x}\right]\,\dd x$ and
\begin{align*}
\int_{\mathbb{R}_{+}}\frac{1}{x}\left[e^{-x/c}-e^{- x}\right]\,\dd x
& = \int_0^\infty\int_1^{1/c} -e^{-ux}\,\dd u\,\dd x\\
& = \int_1^{1/c}\int_0^\infty -e^{-ux}\,\dd x\,\dd u \\
& = \int_1^{1/c}-\frac{1}{u}\,\dd u  = \log c.
\end{align*}
Application of this result to $c=\xi_k/n$ gives $I_2=\alpha\sum_k \left(\tau^n_{k}-\tau^n_{k-1}\right)\log(\xi_k/n)$.
\end{proof}

\section{Observations with piecewise constant volatility}\label{section:pw}

In this section we will provide contraction rates for the posterior distribution in our Bayesian setup, when the true volatility function $\sigma$ is piecewise constant. So, we
consider the process $X^n$ whose true distribution results from $X^n$ being  the solution to
\begin{equation}\label{eq:sde^n}
\dd X^n_t = \frac{1}{n} \sum_{k=1}^K \sigma_k \mathbf{1}_{B_k}(X^n_{t-}) \dd L_t,
\end{equation}
where the constants $\sigma_k$ are assumed to be strictly positive. So $\sigma(x)=\sum_k \sigma_k \mathbf{1}_{B_k}(x)$, in agreement with \eqref{eq:sden}. We assume that the process $X^n$ is observed until the time where it surpasses the fixed and known level $b_K$. We aim at estimating $\sigma$, or equivalently, the sequence $\{\sigma_k\}$, in a consistent way. In order to accomplish this, the observed process has to spend large times in the bins $B_k$, which is effectively the result of the scaling factor $\frac{1}{n}$ for large $n$ in \eqref{eq:sde^n}. In fact, one has that the $\tau^n_k-\tau^n_{k-1}$ are roughly proportional to $n$ for $n\to\infty$, see Proposition~\ref{prop:limittau} below for the precise result. In our Bayesian approach we model the $\sigma_k$ as independent random variables $\xi_k$ and we take inverse gamma distributions as a prior for each of them, that is \(\xi_{k}\sim \ig(\alpha_k,\beta_k)\). Therefore, we have to extend the original probability space to carry the $\xi_k$ as well, taking into account that $L$ and the $\sigma_k$ have to be independent. We can then use Corollary~\ref{cor:lik}, where $\pp^\sigma_T$ is to be interpreted as the conditional law of $X^n$ on $[0,T]$ given the $\xi_k$.

Let $x_1,\ldots,x_K$ denote realisations of $\xi_1,\ldots\xi_K$ and let $X^n$ stand for the path $X_t^n,$  $t\in [0,T]$. With $T=\tau^n_K$, we then have from Corollary~\ref{cor:lik} that the posterior joint density of $(\xi_1,\ldots,\xi_K)$ is given by
\[
\Pi(x_1,\ldots,x_K\mid X^n)\propto\prod_{k=1}^{K}\exp\left(-(n\beta (X^n_{\tau^n_{k}}- X^n_{\tau^n_{k-1}})+\beta_{k})x_{k}^{-1}\right)x_{k}^{-\alpha(\tau^n_{k}-\tau^n_{k-1})-\alpha_{k}-1}.
\]
It follows that the $\xi_k$ are independent under the posterior distribution, and
\begin{equation}\label{eq:igposterior}
\xi_k \mid X^n \sim \operatorname{IG}(\alpha \Delta \tau^n_k + \alpha_k,  n\beta \Delta X_{\tau^n_k} + \beta_k),
\end{equation}
where $\Delta \tau^n_k=\tau^n_{k}-\tau^n_{k-1}$, $\Delta X_{\tau^n_k}=X_{\tau^n_{k}}- X_{\tau^n_{k-1}}$.

\subsection{Result on the overshoot}

To prepare for our first main result, we need a property of the overshoots of $X^n$. The overshoot $\pot^n_k$ is defined as $\pot^n_k:=X^n_{\tau^n_k}-b_k\geq 0$ for $k=0,\ldots,K$. Note that $\pot^n_0=0$. 

\begin{lemma}\label{lemma:overshoot}
Let $\sigma^*_K=\max\{\sigma_1,\ldots,\sigma_K\}$ and $\delta>0$. The probability $\pp(\pot^n_{k} < \delta)$ satisfies the following bound,
\begin{equation}\label{eq:overshoot}
\pp(\pot^n_{k} < \delta) \geq \frac{b_K}{b_K+\delta}(1-\exp(-n\delta\beta/\sigma^*_K)).
\end{equation}
If $\delta_n\to 0$ such that $n\delta_n\to\infty$, then $\pp(\pot^n_{k} > \delta_n) \to 0$.  
\end{lemma}

\begin{proof}
Since $\pot^n_0=0$, inequality \eqref{eq:overshoot} is trivially true for $k=0$, so we let $k\in\{1,\ldots,K\}$ and $0<a<b_k$.
We start by considering the conditional probability $\pp(\pot^n_{k} < \delta \mid X^n_{\tau^n_{k}-} = b_{k} - a)$ and we claim that  
\begin{equation}\label{eq:claim}
\pp(\pot^n_{k} < \delta \mid X^n_{\tau^n_{k}-} = b_{k} - a)=\frac{\nu^n_k([a,a+\delta))}{\nu^n_k([a,\infty))},
\end{equation}
where $\nu^n_k$ is the L\'evy measure of the process $\frac{\sigma_k}{n}L$. Note that $\nu^n_k([a,b]) = \nu([an/\sigma_k, bn/\sigma_k])$ with $\nu$ the L\'evy measure of $L$.
To see \eqref{eq:claim}, 
we argue as follows. 
First we have
\begin{align*}
\pp(\pot^n_{k} < \delta \mid X^n_{\tau^n_{k}-} = b_{k} - a)
& =
\pp(\Delta X^n_{\tau^n_k} <b_k-X^n_{\tau^n_{k-}}+ \delta \mid X^n_{\tau^n_{k}-} = b_{k} - a) \\
& = \pp(\Delta X^n_{\tau^n_k} <a+ \delta \mid X^n_{\tau^n_{k}-} = b_{k} - a).
\end{align*} 
We will now use \cite[Theorem~5.6]{kyprianou14}, that states in terms of densities (which exist here) and in notation adapted to our situation that the random vector $(X^n_{\tau^n_k}-b_k, b_k-X^n_{\tau^n_k-})$ has joint density (for $x>0$, $0<y<b_k$)
\[
f_{X^n_{\tau^n_k}-b_k, b_k-X^n_{\tau^n_k-}}(x,y)=f_U(b_k-y)f_\nu(x+y),
\]
where $f_U$ is the density of the potential function $U$ (in Kyprianou's terminology; it will turn out that the precise form of $f_U$ is not relevant in our context) and $f_\nu$ the density of the L\'evy measure of the process, in our case $\nu^n_k$, as we effectively deal with properties of the process $\frac{\sigma_k}{n}L$. It follows that $b_k-X^n_{\tau^n_k-}$ has marginal density
\[
f_{b_k-X^n_{\tau^n_k-}}(y)=f_U(b_k-y)\nu^n_k([y,\infty)).
\]
Using the change of variables $(X^n_{\tau^n_k}-b_k, b_k-X^n_{\tau^n_k-})\to (\Delta X^n_{\tau^n_k}, b_k-X^n_{\tau^n_k-})$ one gets from the transformation formula that $(\Delta X^n_{\tau^n_k}, b_k-X^n_{\tau^n_k-})$ has density
\[
f_{\Delta X^n_{\tau^n_k}, b_k-X^n_{\tau^n_k-}}(x,y)=f_{X^n_{\tau^n_k}-b_k, b_k-X^n_{\tau^n_k-}}(x-y,y)=f_U(b_k-y)f_{\nu^n_k}(x),
\]
for $x>y$, $0<y<b_k$ and zero elsewhere. It follows that the conditional density of $\Delta X^n_{\tau^n_k}$ given $b_k-X^n_{\tau^n_k-}=a$ is, for $x>a$,
\[
f_{\Delta X^n_{\tau^n_k}\mid b_k-X^n_{\tau^n_k-}=a}(x)=\frac{f_U(b_k-a)f_{\nu^n_k}(x)}{f_U(b_k-a){\nu^n_k}([a,\infty))}=\frac{f_{\nu^n_k}(x)}{\nu^n_k([a,\infty))}.
\]
Hence,
\begin{align*}
\pp(\Delta X^n_{\tau^n_k} <a+ \delta \mid X^n_{\tau^n_{k}-} = b_{k} - a)
& =
\int_a^{a+\delta}\frac{f_{\nu^n_k}(x)}{{\nu^n_k}([a,\infty))}\,\dd x \\
& =
\frac{{\nu^n_k}(a,a+\delta)}{{\nu^n_k}([a,\infty))},
\end{align*}
which proves \eqref{eq:claim}.
Next we show that $g(a):=\frac{F(a + \delta) - F(a)}{1 - F(a)}$ is decreasing as a function of $a$ for $a\geq\eps$, where $\eps>0$ is sufficiently small.
Here $F$ is the `truncated distribution function' of $\nu$, $F(x)=\nu((\eps,x])$ for $x\geq\eps$. Without loss of generality, we scale $F$ such that $F(\infty)=1$. Then, denoting by $f$ the derivative of $F$, one has
\begin{align*}
\frac{\dd}{\dd a} \frac{F(a + \delta) - F(a)}{1 - F(a)} & = \frac{ (1-F(a) ) f(a+ \delta) - (1- F(a + \delta) ) f(a)}{\underbrace{(1 - F(a))^2}_{=:c > 0}}
\\
& = \frac{1}{c}\int_{a}^\infty f(x) f(a+\delta) \dd x -
\frac{1}{c}\int_{a+\delta}^\infty f(x) f(a) \dd x
\\
& = \frac{1}{c}\int_{a}^\infty \left(  f(x) f(a+\delta) - f(x+\delta) f(a) \right) \dd x. 
\end{align*}
With $\nu(\dd x) = f(x)\dd x= \frac{\alpha}{x} \exp(-\beta x)\dd x$ for every $\alpha$ and $\beta$ (which we later replace with $\frac{n\beta}{\sigma_k}$), 
\begin{align*}
\lefteqn{\frac{\dd}{\dd a} \pp(\pot^n_{k} < \delta \mid X^n_{\tau^n_{k}-} = b_{k} - a) }
\\
& = \frac{1}{c}\int_{a}^\infty \left(  \frac{1}{(a+\delta)x} -  \frac{1}{(x+\delta)a}\right)\exp(-\beta (x + \delta + a))  \dd x 
\\
&= \frac{1}{c}\int_{a}^\infty \underbrace{\left( \frac{\delta(a-x)}{a x (x+\delta)(a+\delta)}\right)}_{\le 0}\exp(-\beta (x + \delta + a))  \dd x  \le 0.
\end{align*}
Noting that $g(a)$ is defined for any $a$ away from zero, we now apply the decreasing behaviour of $g$, $g(a) \geq g(b_K)$ as $a<b_k<b_K$, (with $\beta$ replaced with $n\beta/\sigma_k$) to get the lower bound 
\begin{equation}\label{eq:pot1}
\pp(\pot^n_{k} < \delta \mid X^n_{\tau^n_{k}-} = b_{k} - a)=g(a) \geq g(b_K)= \frac{\nu^n_k([b_K,b_K+\delta))}{\nu^n_k([b_K,\infty))}.
\end{equation}
Next we will obtain upper and lower bounds for the fraction $q(n,a,k,\delta)=\frac{\nu^n_k([a,a+\delta))}{\nu^n_k([a,\infty))}$ for any $a$ and $k$. We write the fraction as a fraction of integrals keeping in mind our model with volatility scaled by $n$ and $L$ a gamma process. We will provide bounds on $\nu^n_k([a,b))$ with $a<b\leq\infty$. We use that for $X^n_{t-}\in B_k$ we have to deal with the L\'evy measure $\nu^n_k$ with density
\[
v^n_k(x)=\frac{\alpha}{x}\exp(-n\beta x/\sigma_k).
\]
We compute
\begin{align*}
\nu^n_k([a,b)) & =\int_a^b\frac{\alpha}{x}\exp(-n\beta x/\sigma_k)\,\dd x \\
& =\int_{an\beta/\sigma_k}^{bn\beta/\sigma_k}\frac{\alpha}{y}\exp(-y)\,\dd y \\
& \leq \int_{an\beta/\sigma_k}^{bn\beta/\sigma_k}\frac{\alpha}{an\beta/\sigma_k}\exp(-y)\,\dd y \\
& = \frac{\alpha\sigma_k}{an\beta}(\exp(-an\beta/\sigma_k)-\exp(-bn\beta/\sigma_k)).
\end{align*}
Via a similar argument, we have a lower bound
\[
\nu^n_k([a,b))\geq \frac{\alpha\sigma_k}{bn\beta}(\exp(-an\beta/\sigma_k)-\exp(-bn\beta/\sigma_k)).
\]
Hence, using the lower bound with $b=a+\delta$ and the upper bound with $b=\infty$, we obtain
\begin{equation}\label{eq:q}
q(n,a,k,\delta)\geq \frac{a}{a+\delta}(1-\exp(-n\delta\beta/\sigma_k)).
\end{equation}
Applying \eqref{eq:q} to \eqref{eq:pot1}, we obtain
\[
\pp(\pot^n_{k} < \delta \mid X^n_{\tau^n_{k}-} = b_{k} - a)\geq
\frac{b_K}{b_K+\delta}(1-\exp(-n\delta\beta/\sigma_k)).
\]
As the above lower bound is a decreasing function of $\sigma_k$, we can make it smaller by replacing $\sigma_k$ with $\sigma^*_K$ and obtain
\[
\pp(\pot^n_{k} < \delta \mid X^n_{\tau^n_{k}-} = b_{k} - a)\geq
\frac{b_K}{b_K+\delta}(1-\exp(-n\delta\beta/\sigma^*_K)). 
\]
Write $\mu^n_k$ for the distribution of $X^n_{\tau^n_{k}-}$. Then the unconditional probability $\pp(\pot^n_{k} < \delta)$ can be written as
\begin{align*}
\pp(\pot^n_{k} < \delta) & = \int_{[b_{k-1},b_k)}\pp(\pot^n_{k} < \delta \mid X^n_{\tau^n_{k}-} = x) \,\mu^n_k (\dd x) \\
& \geq \int_{[b_{k-1},b_k)}\frac{b_K}{b_K+\delta}(1-\exp(-n\delta\beta/\sigma^*_K)) \,\mu^n_k (\dd x) \\
& = \frac{b_K}{b_K+\delta}(1-\exp(-n\delta\beta/\sigma^*_K)), 
\end{align*}
which proves \eqref{eq:overshoot}. 
The final assertion on $\pp(\pot^n_{k} < \delta_n)$ for $\delta_n\to 0$ immediately follows from \eqref{eq:overshoot}.
\end{proof}

\subsection{Posterior contraction rate}\label{section:pwbayes}

Statements on convergence in probability refer to the law $\pp=\pp^{\sigma,n}$ of the process satisfying the SDE \eqref{eq:sden}.
We will denote the posterior distribution of the $\xi_k$ by $\Pi_n$ and posterior expectation and variance by $\ee_{\Pi_n}$ and $\pvar$, respectively. Furthermore we write $\Delta b_k:=b_k-b_{k-1}$ for $k\geq 1$.

\begin{lemma}\label{lemma:deltatau}
Let  $\delta_n>0$ and $G^n_k  =\{\pot^n_k>\delta_n\}$ for $k\geq 0$. Then it holds for $t\geq 0$ and $k\geq 1$ that
\[
\pp\Bigl(L_t<\frac{n(\Delta b_k-\delta_n)}{\sigma_k}\Bigr)\pp((G^n_{k-1})^c)
\leq
\pp(\Delta\tau^n_k> t)
\leq
\pp(L_t<\frac{n\Delta b_k}{\sigma_k}).
\]
\end{lemma}

\begin{proof}
Let $t\geq 0$. On the event $\{\Delta\tau^n_k> t\}$ we have $b_{k-1}\leq X^n_{\tau^n_{k-1}}< b_k$ and for $s\in (\tau^n_{k-1},\tau^n_{k}]$, one has $\sigma(X^n_{s-})=\sigma_k$. Hence 
\begin{align}
\pp(\Delta\tau^n_k> t) & = \pp(X^n_{t+\tau^n_{k-1}} < b_k) \nonumber\\
& =\pp\Bigl(X^n_{\tau^n_{k-1}}+\frac{1}{n}\int_{\tau^n_{k-1}}^{t+\tau^n_{k-1}}\sigma(X^n_{s-})\dd L_s<b_k\Bigr) \nonumber\\
& =\pp\Bigl(\sigma_k(L_{t+\tau^n_{k-1}}-L_{\tau^n_{k-1}})<n(b_k-X^n_{\tau^n_{k-1}})\Bigr) \nonumber\\
& =\pp\Bigl(L^{n,k}_t<\frac{n(\Delta b_k-\pot^n_{k-1})}{\sigma_k}\Bigr),\label{eq:tau11}
\end{align} 
where $L^{n,k}_t=L_{t+\tau^n_{k-1}}-L_{\tau^n_{k-1}}$. 
Trivially, this implies the inequality
\begin{equation}\label{eq:lmarkov}
\pp(\Delta\tau^n_k> t) \leq\pp\Bigl(L^{n,k}_t<\frac{n\Delta b_k}{\sigma_k}\Bigr)=\pp\Bigl(L_t<\frac{n\Delta b_k}{\sigma_k}\Bigr),
\end{equation}
where the  equality follows from the strong Markov property of $L$ and stationarity of its increments. This gives the upper bound. Next we consider a lower bound for \eqref{eq:tau11},
\begin{align*}\label{eq:pllt1}
\pp\Bigl(L^{n,k}_t<\frac{n(\Delta b_k-\pot^n_{k-1})}{\sigma_k}\Bigr) & \geq   
\pp\Bigl(\{L^{n,k}_{t}<\frac{n(\Delta b_k-\pot^n_{k-1})}{\sigma_k}\}\cap (G^n_{k-1})^c\Bigr) \\
& \geq \pp\Bigl(\{L^{n,k}_{t}<\frac{n(\Delta b_k-\delta_n)}{\sigma_k}\}\cap (G^n_{k-1})^c\Bigr).
\end{align*}
By independence of $L_{t+\tau^n_{k-1}}-L_{\tau^n_{k-1}}$ and $X^n_{\tau^n_{k-1}}$, the strong Markov property of $L$ and stationarity of its increments, the last probability  is equal to 
\[
\pp\Bigl(L_t<\frac{n(\Delta b_k-\delta_n)}{\sigma_k}\Bigr)\pp((G^n_{k-1})^c),
\]
which is the desired lower bound.
\end{proof}

\begin{prop}\label{prop:limittau}
Let $\Delta\bar\tau^n_k=\frac{n\Delta b_k\beta}{\alpha \sigma_k}$.
For $c_n\to\infty$ such that $c_n n^{-\half}\to 0$ it holds that 
\[
\pp\Bigl((1-\frac{c_n}{\sqrt{n}})\Delta\bar\tau^n_k \leq \Delta\tau^n_k \leq (1+\frac{c_n}{\sqrt{n}})\Delta\bar\tau^n_k\Bigr)\to 1.
\]
\end{prop}

\begin{proof}
We consider
\[
\pp\Bigl(\Delta\tau^n_k \leq \frac{n\Delta b_k\beta}{\alpha \sigma_k} (1+\frac{c_n}{\sqrt{n}})\Bigr)=1-\pp\Bigl(\Delta\tau^n_k> \frac{n\Delta b_k\beta}{\alpha \sigma_k} (1+\frac{c_n}{\sqrt{n}})\Bigr)
\]
and derive an upper bound for it using the lower bound for $\pp(\Delta\tau^n_k<t)$ in Lemma~\ref{lemma:deltatau} and the fact that $L_t$ has the $\Gamma(\alpha t,\beta)$ distribution. One has, with $t=u^n_k:=\frac{n\Delta b_k\beta}{\alpha \sigma_k} (1+\frac{c_n}{\sqrt{n}})$, and in view of of the upper bound in Lemma~\ref{lemma:deltatau},
\[
\pp(\Delta\tau^n_k > u^n_k)
\leq 
\pp\Bigl(L_{u^n_k}<\frac{n\Delta b_k}{\sigma_k}\Bigr).
\]
Note that $\ee L_{u^n_k}=\frac{\alpha}{\beta} u^n_k$ and $\var L_{u^n_k}=\frac{\alpha}{\beta^2} u^n_k$. By the central limit theorem for gamma distributions, 
$\hat L_{u^n_k}=\frac{L_{u^n_k}-\frac{\alpha}{\beta} u^n_k}{\sqrt{\frac{\alpha}{\beta^2} u^n_k}}$ has an asymptotic standard normal distribution.  The probability 
on the right of the above display can be rewritten as 
\[
\pp\Bigl(\hat L_{u^n_k}< \frac{\frac{n\Delta b_k}{\sigma_k}-\frac{\alpha}{\beta} u^n_k}{\sqrt{\frac{\alpha}{\beta^2} u^n_k}}\Bigr).
\]
The term on the right-hand side of the inequality in parentheses is seen to be equal to
\(-\frac{c_n}{\sigma_k}\sqrt{\frac{\beta\sigma_k\Delta b_k}{1+\frac{c_n}{\sqrt{n}}}}.\)
This term tends to minus infinity and so $\pp(\Delta\tau^n_k > u^n_k)\to 0$.
Next we consider $\pp(\Delta\tau^n_k > l^n_k)$ with $l^n_k:=\frac{n\Delta b_k\beta}{\alpha \sigma_k} (1-\frac{c_n}{\sqrt{n}})$ and show that this probability tends to one. First we use the lower bound in Lemma~\ref{lemma:deltatau},
\[
\pp(\Delta\tau^n_k > l^n_k)
\geq 
\pp\Bigl(L_{l^n_k}<\frac{n(\Delta b_k-\delta_n)}{\sigma_k}\Bigr)\pp((G^n_{k-1})^c).
\]
Note that $\pp((G^n_{k-1})^c)\to 1$ by Lemma~\ref{lemma:overshoot}.
As for the previous case, we look at the standardisation $\hat L_{l^n_k}$ of $L_{l^n_k}$ and consider
\[
\pp\Bigl(L_{l^n_k}<\frac{n(\Delta b_k-\delta_n)}{\sigma_k}\Bigr)\geq 
\pp\Bigl(\hat L_{l^n_k}
<
\frac{\frac{n(\Delta b_k-\delta_n)}{\sigma_k}-\frac{\alpha}{\beta} l^n_k}{\sqrt{\frac{\alpha}{\beta^2} l^n_k}}\Bigr).
\]
The right hand side of the inequality in parentheses is seen to be equal to
\[
\frac{1}{\sigma_k}\sqrt{\frac{\beta\sigma_k}{1-\frac{c_n}{\sqrt{n}}}}\left(c_n\sqrt{\Delta b_k}-\delta_n\sqrt{\frac{n}{\Delta b_k}}
\right).
\]
This term tends to plus infinity since $\frac{c_n}{\sqrt{n}}\to 0$ and if we choose, as we do, $\delta_n$ such that $\delta_n\sqrt{n}$ is bounded, then $\pp(\Delta\tau^n_k > l^n_k)\to 0$. 
\end{proof}

\begin{lemma}\label{lemma:mse1}
Let $c_n\to\infty$ such that $c_nn^{-\half}\to 0$.
Then the posterior mean squared error $\ee_{\Pi_n}(\xi_k-\sigma_k)^2 =O(\frac{c_n^2}{n})$ on a set of probability tending to one for all $k=1,\ldots,K$.
\end{lemma}

\begin{proof}
Recall the inverse gamma posterior distribution of the $\xi_k$ as given in \eqref{eq:igposterior}. 
We will first consider the posterior bias
\[
\mathrm{bias}_k=\ee_{\Pi_n}\xi_k-\sigma_k=\frac{n\beta\Delta X^n_{\tau^n_k}+\beta_k}{\alpha\Delta \tau^n_k+\alpha_k-1}-\sigma_k.
\]
and  will provide upper and lower bounds for it. Let $\gamma_k=\frac{\beta \Delta b_k}{\alpha\sigma_k}$  
and $\delta_n\to 0$ with $n\delta_n\to \infty$. In addition to the events $G^n_k$ we need the sets 
\begin{align}
F^n_k & =\{\Delta \tau^n_k >\gamma_k n (1+\frac{ c_n}{\sqrt{n}})\} \\
H^n_k & =\{\Delta\tau^n_k <\gamma_k n (1-\frac{ c_n}{\sqrt{n}})\}.
\end{align} 
It follows from Proposition~\ref{prop:limittau} that $\pp(F^n_k)\to 0$ and $\pp(H^n_k)\to 0$. Furthermore, we need $\delta_n$ such that $n\delta_n\to\infty$. From Lemma~\ref{lemma:overshoot} we obtain $\pp(G^n_k)\to 0$.
A first lower bound is given by
\[
\mathrm{bias}_k\geq \frac{n\beta(\Delta b_k-\pot^n_{k-1})+\beta_k}{\alpha\Delta \tau^n_k+\alpha_k-1}-\sigma_k,
\]
which we split on the events $(F^n_k)^c\cap (G^n_{k-1})^c$ and $F^n_k\cup G^n_{k-1}$. As $\pp(F^n_k\cup G^n_{k-1})\to 0$, we ignore the behaviour on $F^n_k\cup G^n_{k-1}$. On the event $(F^n_k)^c\cap (G^n_{k-1})^c$
we can further bound $\mathrm{bias}_k$ from below by
\[
\frac{n\beta (\Delta b_k-\delta_n)+\beta_k}{n\beta \Delta b_k (1+\frac{ c_n}{\sqrt{n}})/\sigma_k+\alpha_k-1}-\sigma_k,
\]
which we compute further as
\[
\sigma_k\frac{ -\delta_n-\frac{ c_n\Delta b_k}{\sqrt{n}}+\frac{\beta_k-(\alpha_k-1)\sigma_k}{n\beta}}{\Delta b_k (1+\frac{ c_n}{\sqrt{n}})+\frac{(\alpha_k-1)\sigma_k}{n\beta}}.
\]
This is seen to be of order $O(\frac{c_n}{\sqrt{n}})$ provided we choose $\delta_n\sqrt{n}$ bounded, as we can do, still keeping $n\delta_n\to\infty$.
Next we derive an upper bound for the bias.
We trivially have 
\[
\mathrm{bias}_k\leq \frac{n\beta(\Delta b_k+\pot^n_{k})+\beta_k}{\alpha\Delta \tau^n_k+\alpha_k-1}-\sigma_k.
\]
On the set $(H^n_k)^c\cap (G^n_{k})^c$, which has probability tending to one, we can further upper bound this by
\[
\sigma_k\frac{ \delta_n+\frac{ c_n\Delta b_k}{\sqrt{n}}+\frac{\beta_k-(\alpha_k-1)\sigma_k}{n\beta}}{\Delta b_k (1-\frac{ c_n}{\sqrt{n}})+\frac{(\alpha_k-1)\sigma_k}{n\beta}},
\]
which is again of order $O(\frac{c_n}{\sqrt{n}})$ for $\delta_n\sqrt{n}$ bounded. We conclude that $\mathrm{bias}_k$ is of order $O(\frac{c_n}{\sqrt{n}})$ on an event with probability tending to one.
\medskip\\
We move on to the posterior variance and derive upper and lower bounds on it. Recall
\[
\pvar\xi_k=\frac{(n\beta\Delta X^n_{\tau^n_k}+\beta_k)^2}{(\alpha\Delta\tau^n_k+\alpha_k-1)^2(\alpha\Delta\tau^n_k+\alpha_k-2)}.
\]
On the set $(F^n_k)^c\cap (G^n_{k-1})^c$ this is larger than
\[
\frac{(n\beta(\Delta b_k-\delta_n)+\beta_k)^2}{(n\beta \Delta b_k (1+\frac{ c_n}{\sqrt{n}})/\sigma_k+\alpha_k-1)^2(n\beta \Delta b_k (1+\frac{ c_n}{\sqrt{n}})/\sigma_k+\alpha_k-2)},
\]
which is of order $O(\frac{1}{n})$, since $\delta_n\to 0$.
\medskip\\
Next we give an upper bound for the posterior variance on the set $(H^n_k)^c\cap (G^n_{k})^c$, which is
\[
\frac{(n\beta(\Delta b_k+\delta_n)+\beta_k)^2}{(n\beta \Delta b_k (1-\frac{ c_n}{\sqrt{n}})/\sigma_k+\alpha_k-1)^2(n\beta \Delta b_k (1-\frac{ c_n}{\sqrt{n}})/\sigma_k+\alpha_k-2)}.
\]
As for the lower bound, also this bound is of order $O(\frac{1}{n})$.
Combining the properties of posterior bias and variance, we obtain the posterior mean squared error $\ee_{\Pi_n}(\xi_k-\sigma_k)^2$ is of order $O(\frac{c_n^2}{n})$ with probability tending to one. 
\end{proof}

\begin{rem}
The assertion of Lemma~\ref{lemma:mse1} can alternatively be formulated as $\ee_{\Pi_n}(\xi_k-\sigma_k)^2 =O_\pp(\frac{c_n^2}{n})$.
\end{rem} 
Our main result of this section is the following theorem, which says that the posterior contraction rate for estimating $\sigma_k$ is $n^{-1/2}$.

\begin{thm}
\label{thm:contr-lip}
Let $(m_n)$ be any sequence of positive real numbers converging to infinity. Then, for $n\to\infty$,
\[
\Pi_n\Bigl(|\xi_k-\sigma_k|>\frac{m_n}{\sqrt{n}}\Bigr)\to 0 \mbox{ in probability}.
\]
\end{thm}
\begin{proof}
It is sufficient to prove the assertion for $m_n$ increasing to infinity slow enough. For such $m_n$, 
let $c_n\to\infty$ such that $\frac{c_n}{m_n}\to 0$, for instance $c_n=\sqrt{m_n}$. Then also $c_nn^{-\half}\to 0$ and by Chebychev's inequality and Lemma~\ref{lemma:mse1} we have for all $k=1,\ldots, K$, that
\[
\Pi_n(|\xi_k-\sigma_k|>\frac{m_n}{\sqrt{n}})\leq \frac{n}{m_n^2}\ee_{\Pi_n}(\xi_k-\sigma_k)^2 =O(\frac{c_n^2}{m_n^2}),
\]
with probability tending to one.
\end{proof}

\section{H\"older continuous volatility}\label{section:holder}

We consider again the process $X^n$ satisfying \eqref{eq:sden}, but the standing assumption in this section is that $\sigma$ is H\"older continuous, that is, there are constants $H\geq 0$ and $0<\lambda\leq 1$ such that for all $x,y>0$ it holds that $|\sigma(x)-\sigma(y)|\leq H|x-y|^\lambda$.  Moreover, $\sigma$ is assumed to be bounded from below by a positive constant $\underline{\sigma}$.  

Here is some further notation for the present section. 
\begin{itemize}
\item
The number of bins and their width depend on $n$. So we write $B^n_k=(b^n_{k-1},b^n_k]$, $k=1,\ldots,K$, $K=K_n$. We assume equidistant bins. Let $b_K$ be the endpoint of the last bin, assumed to be fixed. We take the other bin boundaries $b^n_k$ as $b^n_k=\frac{b_Kk}{K}$, $k=1,\ldots,K.$ A given $x\in (0,b_K]$ then belongs to bin $B^n_k$ with $k=k_n(x)=\lceil \frac{Kx}{b_K}\rceil$.
\item
$\Delta b^n_k=b^n_k-b^n_{k-1}$. Note that for $x\in B^n_k$ it holds that $|\sigma(x)-\sigma^n_k|\leq H(\Delta b^n_k)^\lambda$ for $\sigma^n_k\in\{\sigma(b^n_{k-1}),\sigma(b^n_k)\}$.
\item
If $x\in B^n_k$, we write $\Delta\bar\tau^n_k(x)=\frac{n\Delta b^n_k\beta}{\alpha \sigma(x)}$. 
\item
Furthermore, we assume the number of bins $K=K_n\asymp n^\kappa$ for $0<\kappa<1$. Then, given also the above assumption on the $b^n_k$ and the definition of $\bar\tau^n_k(x)$, one has $\Delta b^n_k\asymp n^{-\kappa}$ and $\bar\tau^n_k(x)\asymp n^{1-\kappa}$ for all $k$ and $x\in B_k$.
\end{itemize}
We observe the process $X^n$ until it crosses the last bin. It follows from Proposition~\ref{prop:deltataunkasymp} below that the time this happens, $\tau^n_K$, is with high probability of order $cn$ with $c$ upperbounded by $\frac{\beta b_K}{\alpha\underline{\sigma}}$.
\medskip\\
Although $\sigma$ is continuous, we model it in our Bayesian approach as a piecewise constant, that is, as
\begin{equation}\label{modelsigma}
\xi^n(x)=\sum_{k=0}^K\xi_k\one_{B^n_k}(x),
\end{equation}
where the $\xi_k$ are assigned the inverse gamma prior distributions as in Section~\ref{section:pw}.

\subsection{Behaviour of $\Delta\tau^n_k$}

We need a variation on Lemma~\ref{lemma:overshoot}. Let $\pot^n_k=X^n_{\tau^n_k}-b_k$ and define for $\delta_n>0$ and $k=0,\ldots,K$ the events $G^n_k=\{\pot^n_{k}>\delta_n\}$. Note that $G^n_0=\emptyset$.

\begin{lemma}\label{lemma:overshoot2}
Let $\sigma^*=\max\{\sigma(x):0\leq x\leq b_K\}$ and $\delta>0$. For all  $n$, $k=0,\ldots,K$, the probability $\pp(\pot^n_{k} < \delta)$ satisfies the following lower bound
\begin{equation}\label{eq:overshoot2}
\pp(\pot^n_{k} < \delta) \geq \frac{b_K}{b_K+\delta}(1-\exp(-n\delta\beta/\sigma^*)).
\end{equation}
If $\delta_n\to 0$ such that $n\delta_n\to\infty$, then $\pp(G^n_k)=\pp(\pot^n_{k} > \delta_n) \to 0$.  
\end{lemma}

\begin{proof}
As in the proof of Lemma~\ref{lemma:overshoot} we first look at $\pp(\pot^n_{k} < \delta \mid X^n_{\tau^n_{k}-} = b^n_{k} - a)$. This probability depends on the values of $\sigma$ on the bin $B^n_{k-1}$, but for all $x$ (in $B^n_{k-1}$) one has $\sigma(x)\leq \sigma^*$. Hence
\[
\pp(\pot^n_{k} < \delta \mid X^n_{\tau^n_{k}-} = b^n_{k} - a)\geq \pp((\pot^n_{k})^* < \delta \mid X^n_{\tau^n_{k}-} = b^n_{k} - a),
\]
where $(\pot^n_{k})^*$ is the overshoot belonging to the process that has constant volatility $\sigma^*$ on $[\tau^n_{k-1},\tau^n_k)$. We can therefore use the lower bounds of Lemma~\ref{lemma:overshoot} to obtain
\eqref{eq:overshoot2}.
\end{proof}
Next we derive bounds on $\pp(\Delta\tau^n_k> t)$.

\begin{lemma}\label{lemma:deltataunk}
Let $0<x<b_K$ and let $k=k_n(x)$ such that $x\in B^n_k$. Let $\delta_n$ be a sequence of positive numbers. Then, for $\Delta b^n_k$ small enough, it holds that
\[
\pp\Bigl(L_t<\frac{n(\Delta b^n_k-\delta_n)}{\sigma(x)+2H(\Delta b^n_k)^\lambda}\Bigr)\pp((G^n_{k-1})^c)
\leq
\pp(\Delta\tau^n_k> t)
\leq
\pp\Bigl(L_t<\frac{n\Delta b^n_k}{\sigma(x)-2H(\Delta b^n_k)^\lambda}\Bigr).
\]
\end{lemma}
\begin{proof}
Let $t\geq 0$. On the event $\{\Delta\tau^n_k> t\}$ we have $b^n_{k-1}\leq X^n_{\tau^n_{k-1}}< b^n_k$. Hence 
\begin{align*}
\pp(\Delta\tau^n_k> t) & = \pp(X^n_{t+\tau^n_{k-1}} < b^n_k) \\
& =\pp\Bigl(X^n_{\tau^n_{k-1}}+\frac{1}{n}\int_{\tau^n_{k-1}}^{t+\tau^n_{k-1}}\sigma(X^n_{s-})\dd L_s<b^n_k\Bigr) \\
& =\pp\Bigl(\int_{\tau^n_{k-1}}^{t+\tau^n_{k-1}}\sigma(X^n_{s-})\dd L_s<n(b^n_k-X^n_{\tau^n_{k-1}})\Bigr) \\
& =\pp\Bigl(\int_{\tau^n_{k-1}}^{t+\tau^n_{k-1}}\sigma(X^n_{s-})\dd L_s<n(\Delta b^n_k-\pot^n_{k-1})\Bigr).
\end{align*} 
Note that for $s\in (\tau^n_{k-1},\tau^n_{k}]$, one has $X^n_{s-}\in (b^n_{k-1},b^n_k]$. Let $\sigma^n_k\in\{\sigma(b^n_{k-1}),\sigma(b^n_k)\}$, then for $x\in [b^n_{k-1},b^n_k]$, one has $|\sigma(x)-\sigma^n_k|\leq H(\Delta b^n_k)^\lambda$. Hence, with $L^{n,k}_t=L_{t+\tau^n_{k-1}}-L_{\tau^n_{k-1}}$, one has
\[
(\sigma^n_k-H(\Delta b^n_k)^\lambda)L^{n,k}_t\leq\int_{\tau^n_{k-1}}^{t+\tau^n_{k-1}}\sigma(X^n_{s-})\dd L_s\leq (\sigma^n_k+H(\Delta b^n_k)^\lambda)L^{n,k}_t.
\]
It follows that we have the double inequality
\begin{equation}\label{eq:tau1}
\pp\Bigl(L^{n,k}_t<\frac{n(\Delta b^n_k-\pot^n_{k-1})}{\sigma^n_k+H(\Delta b^n_k)^\lambda}\Bigr)
\leq \pp(\Delta\tau^n_k> t) \leq
\pp\Bigl(L^{n,k}_t<\frac{n(\Delta b^n_k-\pot^n_{k-1})}{\sigma^n_k-H(\Delta b^n_k)^\lambda}\Bigr).
\end{equation}
Trivially, this implies the inequality
\[
\pp(\Delta\tau^n_k> t) \leq\pp(L^{n,k}_t<\frac{n\Delta b^n_k}{\sigma^n_k-H(\Delta b^n_k)^\lambda}).
\]
The latter probability is by stationarity of increments and the strong Markov property of $L$ equal to
\[
\pp\Bigl(L_t<\frac{n\Delta b^n_k}{\sigma^n_k-H(\Delta b^n_k)^\lambda}\Bigr).
\]
If we take $x\in B^n_k$, we have for the latter probability the upper bound
\begin{equation}\label{eq:pllt2}
\pp\Bigl(L_t<\frac{n\Delta b^n_k}{\sigma(x)-2H(\Delta b^n_k)^\lambda}\Bigl)
\end{equation}
as desired. Next we consider the lower bound in \eqref{eq:tau1},
\begin{equation}\label{eq:pllt}
p^n_k=\pp\Bigl(L^{n,k}_t<\frac{n(\Delta b^n_k-\pot^n_{k-1})}{\sigma^n_k+H(\Delta b^n_k)^\lambda}\Bigr),
\end{equation}
and proceed to give a further lower bound for it. With $G^n_{k-1}=\{\pot^n_{k-1}>\delta_n\}$ one has
\[
p^n_k\geq \pp\Bigl(\Bigl\{L^{n,k}_{t}<\frac{n(\Delta b^n_k-\pot^n_{k-1})}{\sigma^n_k+H(\Delta b^n_k)^\lambda}\Bigr\}\cap (G^n_{k-1})^c\Bigr).
\]The latter probability is bounded from below by
\[
\pp\Bigl(\Bigl\{L^{n,k}_{t}<\frac{n(\Delta b^n_k-\delta_n)}{\sigma^n_k+H(\Delta b^n_k)^\lambda}\Bigr\}\cap (G^n_{k-1})^c\Bigr).
\]
By independence of $L_{t+\tau^n_{k-1}}-L_{\tau^n_{k-1}}$ and $X^n_{\tau^n_{k-1}}$, stationarity and the strong Markov property of $L$, this is equal to 
\[
\pp\Bigl(L_t<\frac{n(\Delta b^n_k-\delta_n)}{\sigma^n_k+H(\Delta b^n_k)^\lambda}\Bigr)\pp((G^n_{k-1})^c).
\] 
If we take $x\in B^n_k$, we have for the first probability in the display the lower bound
\begin{equation}\label{eq:pllt3}
\pp\Bigl(L_t<\frac{n(\Delta b^n_k-\delta_n)}{\sigma(x)+2H(\Delta b^n_k)^\lambda}\Bigr).
\end{equation}
This concludes the proof.
\end{proof}
Let $0<x<b_K$ and $k=k_n(x)$ such that $x\in B^n_k$. We next present a result, Proposition~\ref{prop:deltataunkasymp}, on the asymptotic behaviour of $\Delta\tau^n_k(x)$, which shows that, with high probability, it is concentrated near $\Delta\bar\tau^n_k(x)$ as introduced above. For the result we need the condition on the bin width, $\Delta b^n_k\asymp n^{-\kappa}$, and require
\begin{equation}\label{eq:condkappa1}
\kappa\geq \frac{1}{2\lambda+1}.
\end{equation}
Along with this condition we let $\delta_n\asymp n^{-\delta}$, and require 
\begin{equation}\label{eq:conddelta1}
\frac{1+\kappa}{2}\leq\delta<1.
\end{equation}

\begin{prop}\label{prop:deltataunkasymp}
Let $x\in (0,b_K)$ and $x\in B^n_k$, for $k=k_n(x)$. Let $\Delta\bar\tau^n_k(x)=\frac{n\Delta b^n_k\beta}{\alpha \sigma(x)}$.
Under conditions \eqref{eq:condkappa1} and \eqref{eq:conddelta1} and for $c_n\to\infty$ such that $c_n n^{-\half(1-\kappa)}\to 0$ it holds that 
\[
\pp\Bigl(\bigl(1-\frac{c_n}{\sqrt{n\Delta b^n_k}}\bigr)\Delta\bar\tau^n_k(x) \leq \Delta\tau^n_k \leq \bigl(1+\frac{c_n}{\sqrt{n\Delta b^n_k}}\bigr)\Delta\bar\tau^n_k(x)\Bigr)\to 1.
\]
\end{prop}

\begin{proof}
We consider for $u^n_k(x):=\frac{n\Delta b^n_k\beta}{\alpha \sigma(x)} (1+\frac{c_n}{\sqrt{n\Delta b^n_k}})$ the probability
\[
\pp(\Delta\tau^n_k \leq u^n_k(x))=1-\pp(\Delta\tau^n_k> u^n_k(x))
\]
and derive a lower bound for it using the upper bound for $\pp(\Delta\tau^n_k<t)$ for $t=u^n_k(x)$ as in Lemma~\ref{lemma:deltataunk}. One has 
\[
\pp(\Delta\tau^n_k > u^n_k(x))
\leq 
\pp\Bigl(L_{u^n_k(x)}<\frac{n\Delta b^n_k}{\sigma^n_k-H(\Delta b^n_k)^\lambda}\Bigr).
\]
Note that $\ee L_{u^n_k(x)}=\frac{\alpha}{\beta} u^n_k(x)$ and $\var L_{u^n_k(x)}=\frac{\alpha}{\beta^2} u^n_k(x)$. 
Hence, by the central limit theorem applied to gamma distributed random variables, $\hat L_{u^n_k(x)}=\frac{L_{u^n_k(x)}-\frac{\alpha}{\beta} u^n_k(x)}{\sqrt{\frac{\alpha}{\beta^2} u^n_k(x)}}$ asymptotically has the standard normal distribution.  The probability 
on the right of the above display is, for large enough $n$ less than $\pp(L_{u^n_k(x)}<\frac{n\Delta b^n_k}{\sigma(x)-2H(\Delta b^n_k)^\lambda})$, which can be rewritten as 
\[
\pp\Bigl(\hat L_{u^n_k(x)}< \frac{\frac{n\Delta b^n_k}{\sigma(x)-2H(\Delta b^n_k)^\lambda}-\frac{\alpha}{\beta} u^n_k(x)}{\sqrt{\frac{\alpha}{\beta^2} u^n_k(x)}}\Bigl).
\]
The right hand side of the inequality in parentheses can be rewritten as
\[
\frac{1}{\sigma(x)-2H(\Delta b^n_k)^\lambda}\sqrt{\frac{\beta\sigma(x)}{1+\frac{c_n}{\sqrt{n\Delta b^n_k}}}}\left(-c_n+\frac{2H}{\sigma(x)}(\Delta b^n_k)^\lambda \sqrt{n\Delta b^n_k}(1+\frac{c_n}{\sqrt{n\Delta b^n_k}})\right).
\]
This term tends to minus infinity if $\frac{c_n}{\sqrt{n\Delta b^n_k}}\to 0$, which is assumed, and if $(\Delta b^n_k)^{\lambda+\half} \sqrt{n}$ is bounded, the latter happens under condition~\eqref{eq:condkappa1}. Under this condition it follows from the central limit theorem that $\pp(\Delta\tau^n_k > u^n_k(x))\to 0$.

Next we consider $\pp(\Delta\tau^n_k > l^n_k(x))$ with $l^n_k(x):=\frac{n\Delta b^n_k\beta}{\alpha \sigma(x)} (1-\frac{c_n}{\sqrt{n\Delta b^n_k}})$ and show that this probability tends to one. First we use the lower bound, taken from Lemma~\ref{lemma:deltataunk},
\[
\pp(\Delta\tau^n_k > l^n_k(x))
\geq 
\pp\Bigl(L_{l^n_k(x)}<\frac{n(\Delta b^n_k-\delta_n)}{\sigma^n_k+H(\Delta b^n_k)^\lambda}\Bigr)\pp((G^n_{k-1})^c).
\]
Lemma~\ref{lemma:overshoot2} says that $\pp((G^n_{k-1})^c)\to 1$ and 
\[
\pp\Bigl(L_{l^n_k(x)}<\frac{n(\Delta b^n_k-\delta_n)}{\sigma^n_k+H(\Delta b^n_k)^\lambda}\Bigr)\geq \pp\Bigl(L_{l^n_k(x)}<\frac{n(\Delta b^n_k-\delta_n)}{\sigma(x)+2H(\Delta b^n_k)^\lambda}\Bigr).
\]
As for the previous case, we look at the standardisation $\hat L_{l^n_k(x)}$ of $L_{l^n_k(x)}$ and consider
\[
\pp\Bigl(L_{l^n_k(x)}<\frac{n(\Delta b^n_k-\delta_n)}{\sigma^n_k+H(\Delta b^n_k)^\lambda}\Bigr)\geq 
\pp\Bigl(\hat L_{l^n_k(x)}
<
\frac{\frac{n(\Delta b^n_k-\delta_n)}{\sigma(x)+2H(\Delta b^n_k)^\lambda}-\frac{\alpha}{\beta} l^n_k(x)}{\sqrt{\frac{\alpha}{\beta^2} l^n_k(x)}}\Bigr).
\]
The right hand side of the inequality in parentheses is seen to be equal to
\[
\frac{1}{\sigma(x)+2H(\Delta b^n_k)^\lambda}\sqrt{\frac{\beta\sigma(x)}{1-\frac{c_n}{\sqrt{n\Delta b^n_k}}}}\left(c_n-\delta_n\sqrt{\frac{n}{\Delta b^n_k}}
-\frac{2H}{\sigma(x)}(\Delta b^n_k)^\lambda \sqrt{n\Delta b^n_k}(1-\frac{c_n}{\sqrt{n\Delta b^n_k}})\right).
\]
This term tends to plus infinity under Condition~\eqref{eq:condkappa1} if $\delta_n\sqrt{\frac{n}{\Delta b^n_k}}$ is bounded, which happens under condition~\eqref{eq:conddelta1}, and $\frac{c_n}{\sqrt{n\Delta b^n_k}}\to 0$, which is assumed. Consequently, by the central limit theorem, $\pp(L_{l^n_k(x)}<\frac{n(\Delta b^n_k-\delta_n)}{\sigma^n_k+H(\Delta b^n_k)^\lambda})\to 1$. 
\end{proof}

\subsection{Posterior contraction rate}

Let $x\in (0,b_K)$ and $k=k_n(x)$ such that $x\in B^n_k$. The $\xi_k$  corresponding to $x$ is, as it depends on $n$, also denoted $\xi^n(x)$ instead of $\xi_k$.

As in  Section~\ref{section:pwbayes}, we consider the posterior mean squared error $\ee_{\Pi_n}(\xi_k-\sigma_k)^2$, which we analyse through the corresponding posterior bias and variance. Then the bias of the posterior mean for $x\in B^n_k$ is
\begin{equation}
\label{eq:bias}
\mathrm{bias}(x)=\ee_{\Pi_n}\xi_k-\sigma(x)=\frac{n\beta(X^n_{\tau^{n}_{k}}- X^n_{\tau^{n}_{k-1}})+\beta_k}{\alpha\Delta \tau^n_k+\alpha_k-1}-\sigma(x).
\end{equation}
The posterior variance is
\[
\pvar\xi_k=\frac{(n\beta\Delta X^n_{\tau^n_k}+\beta_k)^2}{(\alpha\Delta\tau^n_k+\alpha_k-1)^2(\alpha\Delta\tau^n_k+\alpha_k-2)}. 
\]
As before we give upper and lower bounds for posterior bias and variance. To do so we need, along with the already introduced events $G^n_k=\{\pot^n_k>\delta_n\}$, the events $F^n_k=\{\Delta\tau^n_k >\Delta\bar\tau^n_k(x) (1+\frac{c_n}{\sqrt{n\Delta b^n_k}})\}$ and $H^n_k=\{\Delta\tau^n_k <\Delta\bar\tau^n_k(x)  (1-\frac{c_n}{\sqrt{n\Delta b^n_k}})\}$  for $c_n\to\infty$ (arbitrarily slowly). We know from Proposition~\ref{prop:deltataunkasymp} that $\pp(F^n_k)\to 0$ and $\pp(H^n_k)\to 0$, and from Lemma~\ref{lemma:overshoot2} that $\pp(G^n_k)\to 0$. 

\begin{lemma}\label{lemma:mse2}
Assume the model with piecewise constant volatility $\xi(x)$ as given by \eqref{section:pwbayes} whereas the true volatility function $x\mapsto \sigma(x)$ is H\"older continuous and bounded from below by a strictly positive constant $\underline\sigma$.
Assume $\Delta b^n_k\asymp n^{-\kappa}$ and condition \eqref{eq:condkappa1}. 
Let $c_n\to\infty$.
Then the posterior mean squared error $\ee_{\Pi_n}(\xi^n(x)-\sigma(x))^2 =O(\frac{c_n^2}{n\Delta b^n_k})$ for all $k=1,\ldots,K$, uniformly in $x\in [0,b_K]$, with probability tending to one. That is,  
$$
\sup_{x\in [0,b_K]}\ee_{\Pi_n}(\xi^n(x)-\sigma(x))^2 =O\Bigl(\max_k\frac{c_n^2}{n\Delta b^n_k}\Bigr)
$$ 
with probability tending to one. 
\end{lemma}

\begin{proof}
Let $x\in B^n_k$.
We consider the bias \eqref{eq:bias} first. The bias can be split into its behaviour on the sets $(F^n_k)^c\cap (G^n_{k-1})^c$ and $F^n_k\cup G^n_{k-1}$. As $\pp(F^n_k\cup G^n_{k-1})\to 0$, we only have to analyze what happens on $(F^n_k)^c\cap (G^n_{k-1})^c$.

We give upper and lower bounds for this bias. We start with a first lower bound. As $\Delta X^n_{\tau^n_k}=\Delta b^n_k+\pot^n_k-\pot^n_{k-1}\geq \Delta b^n_k-\pot^n_{k-1}$, and recalling that $\pot^n_{k-1}<\delta_n$ on $(G^n_{k-1})^c$, we obtain that on  $(F^n_k)^c\cap (G^n_{k-1})^c$ one has
\[
\mathrm{bias}(x)\geq \frac{n\beta (\Delta b^n_k-\delta_n)\one_{\{\Delta\tau^n_k>0\}}+\beta_k}{n\beta\Delta b^n_k (1+\frac{c_n}{\sqrt{n\Delta b^n_k}})/\sigma(x)+\alpha_k-1}-\sigma(x),
\]
of which the right hand side can be rewritten as
\[
\sigma(x)\frac{-\frac{\delta_n}{\Delta b^n_k}-\frac{c_n}{\sqrt{n\Delta b^n_k}}+\frac{\beta_k-(\alpha_k-1)\sigma(x)}{\beta n\Delta b^n_k}}{1+\frac{c_n}{\sqrt{n\Delta b^n_k}}+(\alpha_k-1)\frac{\sigma(x)}{\beta n\Delta b^n_k}}.
\]
This term is of order $\frac{c_n}{\sqrt{n\Delta b^n_k}}$ if $\frac{\delta_n}{\Delta b^n_k}\sqrt{n\Delta b^n_k}$ stays bounded, which happens for $\delta_n=n^{-\delta}$ under condition \eqref{eq:conddelta1}. By continuity of $\sigma$ on $[0,b_K]$ this bound is uniform in $x$. Next we turn to an upper bound for the bias. Now we consider the bias on the events $(H^n_k)^c\cap (G^n_k)^c$ and $H^n_k\cup G^n_k$. As $\pp(H^n_k\cup G^n_k)\to 0$, we can ignore the bias on the latter event.
Using $\Delta X^n_{\tau^n_k}=\Delta b^n_k+\pot^n_k-\pot^n_{k-1}\leq \Delta b^n_k+\pot^n_k$, we have on the set $(H^n_k)^c\cap (G^n_k)^c$
\[
\mathrm{bias}(x)\leq \frac{n\beta (\delta_n +\Delta b^n_k)+\beta_k}{n\beta \Delta b^n_k (1-\frac{c_n}{\sqrt{n\Delta b^n_k}})/\sigma(x)+\alpha_k-1}-\sigma(x),
\]
whose right hand side becomes
\[
\sigma(x)\Bigl(\frac{1+\frac{\delta_n}{\Delta b^n_k}+\frac{\beta_k}{n\beta\Delta b^n_k}}{1-\frac{c_n}{\sqrt{n\Delta b^n_k}}+\frac{(\alpha_k-1)\sigma(x)}{n\beta\Delta b^n_k}}-1\Bigr),
\]
and that is equal to
\[
\sigma(x)\frac{\frac{\delta_n}{\Delta b^n_k}+\frac{c_n}{\sqrt{n\Delta b^n_k}}+\frac{\beta_k-(\alpha_k-1)\sigma(x)}{n\beta\Delta b^n_k}}{1-\frac{c_n}{\sqrt{n\Delta b^n_k}}+\frac{(\alpha_k-1)\sigma(x)}{n\beta\Delta b^n_k}}.
\]
Similar to what we have seen for the lower bound of the bias, also the upper bound is of order $\frac{c_n}{\sqrt{n\Delta b^n_k}}$ under condition \eqref{eq:conddelta1}, and uniform in $x$.
Summarising, under the stipulated conditions, we obtain that $\mathrm{bias}(x)$ is of order $\frac{c_n}{\sqrt{n\Delta b^n_k}}$ on a set with probability tending to one.
\medskip\\
We move on to the posterior variance of $\xi_k$ for $x\in B^n_k$, 
\[
\pvar\xi_k=\frac{(n\beta\Delta X^n_{\tau^n_k}+\beta_k)^2}{(\alpha\Delta\tau^n_k+\alpha_k-1)^2(\alpha\Delta\tau^n_k+\alpha_k-2)}, 
\]
for which we will derive upper and lower bounds as well. Paralleling the  computations for the bias, we have on the event $(F^n_k)^c\cap (G^n_{k-1})^c$ the immediate lower bound
\[
\pvar\xi_k  \geq\frac{(n\beta (\Delta b^n_k-\delta_n)+\beta_k)^2}{(n\beta\Delta b^n_k (1+\frac{c_n}{\sqrt{n\Delta b^n_k}})/\sigma(x)+\alpha_k-1)^2(n\beta\Delta b^n_k (1+\frac{c_n}{\sqrt{n\Delta b^n_k}})/\sigma(x)+\alpha_k-2)},
\]
where the right hand side equals
\[
\frac{\sigma(x)^3}{\beta n\Delta b^n_k}
\frac{(1-\frac{\delta_n}{\Delta b^n_k}+\frac{\beta_k}{\beta n\Delta b^n_k})^2}{(1+\frac{c_n}{\sqrt{n\Delta b^n_k}}+\frac{\sigma(x)(\alpha_k-1)}{\beta n\Delta b^n_k})^2(1+\frac{c_n}{\sqrt{n\Delta b^n_k}}+\frac{\sigma(x)(\alpha_k-2)}{\beta n\Delta b^n_k})}.
\]
This is obviously of order $O(\frac{1}{n\Delta b^n_k})$, as $\frac{\delta_n}{\Delta b^n_k}\to 0$.
Next we give an upper bound for the posterior variance, for which we only consider what happens on $(H^n_k)^c\cap (G^n_k)^c$. On that event one has 
\[
\pvar\xi_k\leq\frac{(n\beta (\Delta b^n_k+\delta_n)+\beta_k)^2}{(n\beta \Delta b^n_k (1-\frac{c_n}{\sqrt{n\Delta b^n_k}})/\sigma(x)+\alpha_k-1)^2(n\beta \Delta b^n_k (1-\frac{c_n}{\sqrt{n\Delta b^n_k}})/\sigma(x)+\alpha_k-2)}.
\]
One sees that this quantity is of order $O(\frac{1}{n\Delta b^n_k})$, if $\frac{\delta_n}{\Delta b^n_k}$ tends to zero, which happens under the condition \eqref{eq:conddelta1} for $\delta_n=n^{-\delta}$.
Combining the two results on the bounds, we conclude that the posterior variance is of order $O(\frac{1}{n\Delta b^n_k})=O(n^{-1+\kappa})$ with probability tending to one. As for the bias, also this order bound is uniform in $x$.
As a last step, by the above established properties of posterior bias and variance, we obtain the posterior mean squared error $\ee_{\Pi_n}(\xi^n(x)-\sigma(x))^2$ is of order $O(\frac{c_n^2}{n\Delta b^n_k})=O(\frac{c_n^2}{n^{1-\kappa}})$ with probability tending to one. Again, this order bound is uniform in $x$ (and $k$). 
\end{proof}
Here is the main result of this section, which says that the contraction rate of the posterior distribution is (at least) $n^{-\lambda/(2\lambda+1)}$, $\lambda$ being the H\"older exponent of $\sigma$.

\begin{thm}
\label{thm:rates}
Assume the model with  volatility  \eqref{modelsigma} whereas the true volatility function $x\mapsto \sigma(x)$ is H\"older continuous of order $\lambda\leq 1$ and bounded from below. Let the bin sizes shrink proportional to $n^{-\frac{1}{2\lambda+1}}$, and let $(m_n)$ be any sequence of real numbers (arbitrarily slowly) diverging to infinity. Then, for $n\to\infty$ 
\[
\sup_{x\in [0,b_K]}\Pi_n\Bigl(|\xi^n(x)-\sigma(x)|>m_n n^{-\frac{\lambda}{2\lambda+1}}\Bigr)\to 0 \mbox{ in probability}.
\]
\end{thm}

\begin{proof}
Let $x\in B^n_k$ and $c_n\to\infty$ such that $\frac{c_n}{m_n}\to 0$, for instance $c_n=\sqrt{m_n}$. By Chebychev's inequality and Lemma~\ref{lemma:mse2}, for $\Delta b^n_k\asymp n^{-\kappa}$ with  $\kappa$ such that  \eqref{eq:condkappa1} is satisfied, we have uniformly in $x$
\[
\Pi_n\Bigl(|\xi^n(x)-\sigma(x)|>m_nn^{-\frac{\lambda}{2\lambda+1}}\Bigr)\leq \frac{n^{\frac{2\lambda}{2\lambda+1}}}{m_n^2}\ee_{\Pi_n}(\xi^n(x)-\sigma(x))^2 =O\Bigl(\frac{n^{\frac{2\lambda}{2\lambda+1}}}{m_n^2}\frac{c_n^2}{n\Delta b^n_k}\Bigr),
\]
with probability tending to one. The choice $\kappa=\frac{1}{2\lambda+1}$ satisfies~\eqref{eq:condkappa1} and $\Delta b^n_k\asymp n^{-\frac{1}{2\lambda+1}}$ is assumed. Hence, we see that the order bound becomes $O(\frac{c_n^2}{m_n^2})$, which tends to zero.
\end{proof}

\subsection{Lower bounds}
In this section we show that the contraction rates in Theorem~\ref{thm:rates} are minimax optimal. 
Let \(\Sigma(\lambda,L)\) denote the H\"older class of functions \(f\) on \([0,1]\) satisfying 
\begin{eqnarray*}
|f(x)-f(y)|\leq L|x-y|^\lambda,\quad x,y\in [0,1].
\end{eqnarray*}
We endow \(\Sigma(\lambda,L)\) with the supnorm, denoted $\|\cdot\|_\infty$.
Denote by $\pp^{\sigma,n}_T$ the law of the process $(X_{t}^{n})_{t\in[0,T]}$
solving the Levy-driven SDE, similar to \eqref{eq:sdey},
\[
\dd X_{t}^{n}=\sigma(X_{t-}^{n})\,\dd L_{t}^{n},\quad X_{0}^{n}=0,
\]
where $L^n$ is a gamma process with a L\'evy density  
\[
v_n(x)=\frac{n\alpha}{x}\exp(-n\beta x).
\]
We shall prove the following statement. 
\begin{prop}
There are constants \(c_0,c_1>0\) not depending on \(n\) such that 

\begin{eqnarray}
\label{eq:lower}
\liminf_{n\to \infty}\inf_{\widehat\sigma_n}\sup_{\sigma \in \Sigma(\lambda,L)}\pp^{\sigma,n}_T\Bigl(\|\sigma-\widehat\sigma_n\|_\infty\geq c_0n^{-\lambda/(1+2\lambda)}\Bigr)\geq c_1,
\end{eqnarray}
where the infimum is taken over all estimators \(\widehat\sigma_n,\) that is, all measurable functions of the path  \(X_{t}^{n},\) \(t\in [0,T].\)
\end{prop} 
It is well known that the posterior cannot converge at a rate faster than the optimal rate of convergence for point estimators, see \cite{ghosal2000convergence}, page 507. Since we have a lower bound and it gives the rate matching our posterior contraction rate, our Bayesian approach is optimal from the frequentist point of view.

\begin{proof}
Our strategy is to follow the approach as in Chapter~2 of \cite{tsybakov2008introduction}. Therefore, our goal is to establish
\begin{equation}\label{eq:tsyb}
\inf_{\hat\sigma_n}\sup_{\sigma \in \Sigma(\lambda,L)}\pp^{\sigma,n}_T (\|\sigma-\widehat\sigma_n\|_\infty\geq s_n) \geq \half(1-V(\pp^{\sigma_0,n}_T,\pp^{\sigma_1,n}_T)),
\end{equation}
where $\sigma_0,\sigma_1$ are two distinct elements of $\Sigma(\lambda,L)$ such that 
\begin{equation}\label{eq:2sn}
\|\sigma_0-\sigma_1\|_\infty\geq 2 s_n, 
\end{equation}
and $V(\pp^{\sigma_0,n}_T,\pp^{\sigma_1,n}_T)$ is the total variation distance between $\pp^{\sigma_0,n}_T$ and $\pp^{\sigma_1,n}_T$. The $s_n$ denotes the desired convergence rate, in our case we aim at $s_n\asymp n^{-\lambda/(2\lambda+1)}$. The result in \eqref{eq:tsyb} results from the exposition in Section~2.2 and Theorem~2.2(i)  together with its proof in \cite{tsybakov2008introduction}. 

We first select $\sigma_0$ and $\sigma_1$.
Fix $h\in(0,1]$ and set 
\[
\sigma_{0}(x)\equiv1,\quad\sigma_{1}(x)=1-h^{\lambda}\psi(x/h),
\]
where $\psi\in \Sigma(\lambda,1/2)$ is a nonnegative monotone decreasing function supported on $[0,1]$ satisfying
$\|\psi\|_{\infty}\leq1,$ $\psi(x)>0$ for $x\in[0,1/2].$ One can take, for example,
\begin{eqnarray*}
\psi(u)=a\exp\left(-\frac{1}{1-u^2}\right)\one_{[0,1]}(u)
\end{eqnarray*}
for \(a>0\) small enough. In fact, taking $a=1/e$ for this choice of $\psi$ we have $\|\psi\|_\infty=\psi(0)=1$.
Note furthermore \(\psi\in \Sigma(1,1/2)\), implying \(\psi\in \Sigma(\lambda,1/2)\) for any \(\lambda\leq 1\) and \(\sigma_{1}\in \Sigma(\lambda,1/2).\)  
Moreover, now \( \|\sigma_0-\sigma_1\|_\infty=  h^\lambda\). Later we will choose $h$ and $s_n$ such that $h^\lambda\geq 2s_n$ in order that \eqref{eq:2sn} holds.

We will next show that $V(\pp^{\sigma_0,n}_T,\pp^{\sigma_1,n}_T)$ is eventually less than some constant $v<1$, after which we can choose $c_1=\half(1-v)$ in \eqref{eq:tsyb} to obtain \eqref{eq:lower}.
The
total variation distance $V$ between the laws $\pp^{\sigma_{0},n}_T$
and $\pp^{\sigma_{1},n}_T$ satisfies (see \cite{kabanov1986variation})
\begin{align*}
V^{2}(\pp^{\sigma_{0},n}_T,\pp^{\sigma_{1},n}_T) & \leq 16\,\mathbb{E}_{\pp^{\sigma_{0},n}_T}\left[\int_{0}^{T}\int_{0}^{\infty}(\sqrt{Y(t,x)}-1)^{2}v_{n}(x)\,\dd x\,\dd t\right],\\
Y(t,x) & =\frac{1}{\sigma_{1}(X_{t-})}v_{n}\Bigl(\frac{x}{\sigma_{1}(X_{t-})}\Bigr)/v_{n}(x).
\end{align*}
Using the inequality $1-e^{-x}\leq x$ holding for all $x>0,$ we get
\begin{align}
V^2(\pp^{\sigma_{0},n}_T,\pp^{\sigma_{1},n}_T) & \leq16\, \mathbb{E}_{\pp^{\sigma_{0},n}_T}\left[\int_{0}^{T}\int_{0}^{\infty}\left(\frac{xn\beta(1-\sigma_{1}^{-1}(X_{t-}))}{2}\right)^{2}v_{n}(x)\,\dd x\,\dd t\right]\nonumber\\
  & \lesssim\mathbb{E}_{\pp^{\sigma_{0},n}_T}\left[\int_{0}^{T}\int_{0}^{\infty}n^{2}h^{2\lambda}\psi^{2}(X_{t-}/h)x^{2}v_{n}(x)\,\dd x\,\dd t\right]\nonumber\\
 & \lesssim n^{2}h^{2\lambda}\mathbb{E}_{\pp^{\sigma_{0},n}_T}\left[\int_{0}^{T}\psi^{2}(X_{t-}/h)\,\dd t\right]\int_{0}^{\infty}x^{2}v_{n}(x)\,\dd x\nonumber\\
 & \lesssim nh^{2\lambda}\mathbb{E}_{\pp^{\sigma_{0},n}_T}\left[\int_{0}^{T}\psi^{2}(X_{t-}/h)\,\dd t\right], \label{eq:vnh}
\end{align}
where here and in the sequel $\lesssim$ for inequality up to a
constant depending on \(\alpha\) and \(\beta\).
Furthermore
\begin{align*}
\mathbb{E}_{\pp^{\sigma_{0},n}_T}\left[\int_{0}^{T}\psi^{2}(X_{t-}/h)\,\dd t\right] & =\int_{0}^{T}\int_{0}^{\infty}\left[\psi^{2}(z/h)p_{L_{t}^{n}}(z)\right]\,\dd z\,\dd t,
\end{align*}
where \(p_{L_{t}^{n}}\) is the density of \(L_{t}^{n}.\)
Using  well-known results for the Gamma function, see e.g.\ the sharp version of the Stirling formula of Theorem~1.6 in \cite{batir2008inequalities}, we have for $t\geq2/(n\alpha)$
\begin{align*}
p_{L_{t}^{n}}(z) & =\frac{(n\beta)^{n\alpha t}}{\Gamma(n\alpha t)}z^{n\alpha t-1}e^{-n\beta z}
 \lesssim\frac{1}{z}\text{\ensuremath{\sqrt{\frac{n\alpha t}{2\pi}}}}\exp\left\{ n\alpha t\left(\log\left(\frac{\beta z}{\alpha t}\right)+1-\frac{\beta z}{\alpha t}\right)\right\}.
\end{align*}
 Fix some $0<\delta<1/2$ and consider the integral 
\begin{align*}
\int_{0}^{\infty}\left[\psi^{2}(z/h)p_{L_{t}^{n}}(z)\right]\,\dd z & \lesssim\int_{0}^{\infty}\frac{1}{y}\text{\ensuremath{\sqrt{\frac{n\alpha t}{2\pi}}}}\psi^{2}(\alpha ty/\beta h)\,e^{n\alpha t(1-y+\log(y))}\,\dd y\\
 & =\int_{|y-1|\leq\delta}\frac{1}{y}\text{\ensuremath{\sqrt{\frac{n\alpha t}{2\pi}}}}\psi^{2}(\alpha ty/\beta h)\,e^{n\alpha t(1-y+\log(y))}\,\dd y\\
 & \quad+\int_{y>1+\delta}\frac{1}{y}\text{\ensuremath{\sqrt{\frac{n\alpha t}{2\pi}}}}\psi^{2}(\alpha ty/\beta h)\,e^{n\alpha t(1-y+\log(y))}\,\dd y\\
 & \quad+\int_{y<1-\delta}\frac{1}{y}\text{\ensuremath{\sqrt{\frac{n\alpha t}{2\pi}}}}\psi^{2}(\alpha ty/\beta h)\,e^{n\alpha t(1-y+\log(y))}\,\dd y\\
 & =:I_{1}+I_{2}+I_{3}.
\end{align*}
Since \(1-y+\log(y)\leq -c(1-y)^2\) for some \(c=c(\delta)>0\) if \(|1-y|\leq\delta<1/2,\) we have with \(z=\sqrt{n\alpha t}(y-1),\)
\begin{align*}
I_{1} &  \lesssim\int_{-\delta}^{\delta}\psi^{2}(\alpha t(1+z/\sqrt{n\alpha t})/\beta h)e^{-cz^{2}}\,\dd z, \\
I_{2} & \lesssim\sqrt{n\alpha t}\,e^{n\alpha t(-\delta+\log(1+\delta))}, \\ 
I_{3} & \lesssim\sqrt{n\alpha t}\,e^{n\alpha t(\delta+\log(1-\delta))}.
\end{align*}
Hence for any $t_{0}\geq2/(n\alpha)$ we derive, using the above estimates of $I_1,I_2,I_3$, 
\begin{align}
\int_{0}^{T}\int_{0}^{\infty}\left[\psi^{2}(z/h)p_{L_{t}^{n}}(z)\right]\,\dd z\,\dd t  
& =\int_{0}^{t_0}\int_{0}^{\infty}\left[\psi^{2}(z/h)p_{L_{t}^{n}}(z)\right]\,\dd z\,\dd t \nonumber\\
& \qquad +
 \int_{t_0}^{T}\int_{0}^{\infty}\left[\psi^{2}(z/h)p_{L_{t}^{n}}(z)\right]\,\dd z\,\dd t
\nonumber\\
&\lesssim t_{0}+h\int_{-\delta}^{\delta}\int_{0}^{T/h}\psi^{2}\left(\frac{\alpha s}{\beta}+\frac{z}{\beta}\sqrt{\frac{\alpha s}{nh}}\right)\,\dd s\,\dd z\nonumber\\
 & \qquad +\sqrt{n\alpha t_{0}}\left(e^{n\alpha t_{0}(-\delta+\log(1+\delta))}+e^{n\alpha t_{0}(\delta+\log(1-\delta))}\right),\label{eq:t0h}
\end{align}
where for the last term it is used that the function \(x\mapsto\sqrt{x}e^{-xa}\) with \(a>0\) is monotone decreasing  for \(x\ge 1/(2a)\). 
The double integral in \eqref{eq:t0h} is bounded in $n,h$ as follows from  
\begin{eqnarray*}
\int_{0}^{\infty}\psi^{2}\left(\frac{\alpha s}{\beta}+\frac{z}{\beta}\sqrt{\frac{\alpha s}{nh}}\right)\,\dd s\leq 1/\alpha+\int_{1/\alpha}^{\infty}\psi^{2}\left(\frac{\alpha s}{\beta}(1-\delta)\right)\,\dd s, 
\end{eqnarray*}
where we used that $\phi$ is bounded by 1, increasing and that in this integral $|z|\leq \delta$. Choosing $t_{0}= o(h)$ and such that $nt_0\to\infty$,
one sees that the term with $h$ in \eqref{eq:t0h} is the dominating term. Recalling \eqref{eq:vnh}, we subsequently take 
$h=c n^{-1/(1+2\lambda)}$ for a small enough constant \(c>0,\) to derive for some $0<v<1$
\[
V^2(\pp^{\sigma_{0},n}_T,\pp^{\sigma_{1},n}_T)\leq v^2
\]
for all \(n\) large enough. With this choice of $h$ we take $s_n=\half c^\lambda n^{-\lambda/(1+2\lambda)}$. Then \eqref{eq:2sn} is satisfied, and from \eqref{eq:tsyb}  we arrive at \eqref{eq:lower} with $c_0=\half c^\lambda$ and $c_1=\half(1-v)$, both strictly positive.
\end{proof}

\section{Real data example}
\label{sec:real}

The North Greenland Ice Core Project (NGRIP) obtained from drilling through arctic ice an oxygen isotope record reaching 120\,000 years into the past beyond the last glacial (\cite{northgreenlandicecoreprojectmembers2007ymoo}.)
Figure~\ref{fig:dataicecore} shows measurements of the indicator  $\delta_{18}\mathrm{O}$ derived from oxygen isotope measurements at times $t = 0, 50, \dots$ in $\Delta t = 50$ year intervals, with in total $n =2459$ observations. In geological scales such a  $\Delta t$ can be considered small.
The oxygen isotope record is a proxy for past temperature, and the data shows characteristic sudden changes in global temperature, a topic of urgent relevance. It has been suggested  to model the NGRIP data as a realisation of a stochastic differential equation with solution $Y$, also as a L\'evy-driven SDE to account for the heavier tailed noise, see \cite{ditlevsen1999observation}. An in-depth study is given in \cite{Carson2019}.

Here we are interested in estimating the volatility $\varsigma$ of the process $Y_t$, observed on the equidistant time grid $0, \Delta t, \ldots, n\Delta t$ with $n$ observations. Estimating the volatility is an important step in data assimilation and inference tasks related to rapid temperature transitions, for example during so called Dansgaard-Oeschger events.

For a particular realisation $y$ of $Y$ and a fixed $\Delta t$, the realised quadratic variation process over a time grid with step size $\Delta t$ can be defined as
$q_{t+\Delta t} - q_t = \Delta q_t$ with $\Delta q_t := (\Delta y_t)^2= (y_{t + \Delta t} -y_{t})^2$, $q_0 = 0$. 
See Figure~\ref{fig:obsicecore} for a visualisation and note the resemblance  (up to scaling) of this figure with Figure~\ref{fig:ex1obs} in Example~\ref{exam:1}.
We remark that the realised quadratic variation process of a diffusion process can be considered as a measure of intrinsic progress of time, also referred to as internal clock. 

For motivation of the model proposed below, suppose that $Y$ is a diffusion process satisfying $\dd Y_t= 
\varsigma_t\dd W_t$, with $W$ a Brownian motion. For small $\Delta t$ one has $(\Delta Y_t)^2\approx \varsigma_t^2 (W_{t+\Delta t}-W_t)^2$, which has (given the past up to time $t$) a  $\operatorname{Gamma}(\half, \frac{1}{2\Delta t \varsigma_t^2})$ distribution. 

Consider next the L\'{e}vy-driven SDE
\begin{equation}\label{eq:realizedvola1}
\dd X_t = c \varsigma^2_t \, \dd L_t,
\end{equation}
 where $L$ is a gamma process with parameters $\alpha$, $\beta$ to be specified shortly. For $X$ solving  \eqref{eq:realizedvola1} we have that $\Delta X_t\approx c\varsigma^2_t\Delta L_t$, which is, conditional on the past up to time $t$, a $\operatorname{Gamma}(\alpha\Delta t, \beta/(c \varsigma_t^2))$ random variable. 
Then with $\alpha = \frac{1}{2\Delta t}$ and $\beta = \frac{c}{2\Delta t }$, the conditional distributions of $\Delta X_t$ and $(\Delta Y_t)^2$ are approximately gamma with the same parameters, for any choice of $c>0$. We used $c =  n\Delta t /q_{n \Delta t}$, which implies $\ee L_{n\Delta t} =  
q_{n\Delta t}$.
With this in mind,  we model $q_{i \Delta t}$ as observation of a realisation of the continuous time process 
\begin{equation}\label{eq:realizedvola}
\dd X_t = \sigma(X_{t-}) \dd L_t.
\end{equation}
with unknown $\sigma$. As both $\alpha$ and $\beta$ are proportional to $1/\Delta t \asymp n$, this corresponds to the regime detailed after \eqref{eq:sdey} with parameters of the driving gamma process proportional to $n$, which suggests that our asymptotic results are practically relevant for this problem. Here, we choose to model the slope $\sigma$ of the curve of realized quadratic variation, see in Figure~\ref{fig:obsicecore}, as a
function of $Y_t = \int_0^t  c \varsigma^2_s \, \dd L_s$, using the monotonicity of the realisation. 

In estimating $\sigma$ with \eqref{eq:realizedvola}, using terminology common in the financial literature, we estimate volatility of $\delta_{18}\rm{O}$ as function of the time measured by the intrinsic (or business) clock. To perform the statistical analysis with our approach we
 set $K = 20$, which appears to be a good compromise in terms of bias-variance trade-off. 
We take equidistant bins over the range of observations of $q_t$. Note that a fixed resolution in space implies a variable resolution in time, with larger bins at times where the increments of $X$ are small.
Then the posterior is determined by choosing the prior parameters; here we took weakly informative parameters $\alpha_k = \beta_k \equiv 0.1$. Figure~\ref{fig:posticecore} gives the marginal posterior band for $\sigma$. The figure shows  that in general with higher intrinsic age $X_t$, the volatility of the measurements decreases. This phenomenon has been noted before, and a relation to aging processes in the ice has been suggested, highlighting two periods of unusual activity.

\begin{figure}
\begin{center}
\includegraphics[width=0.7\linewidth]{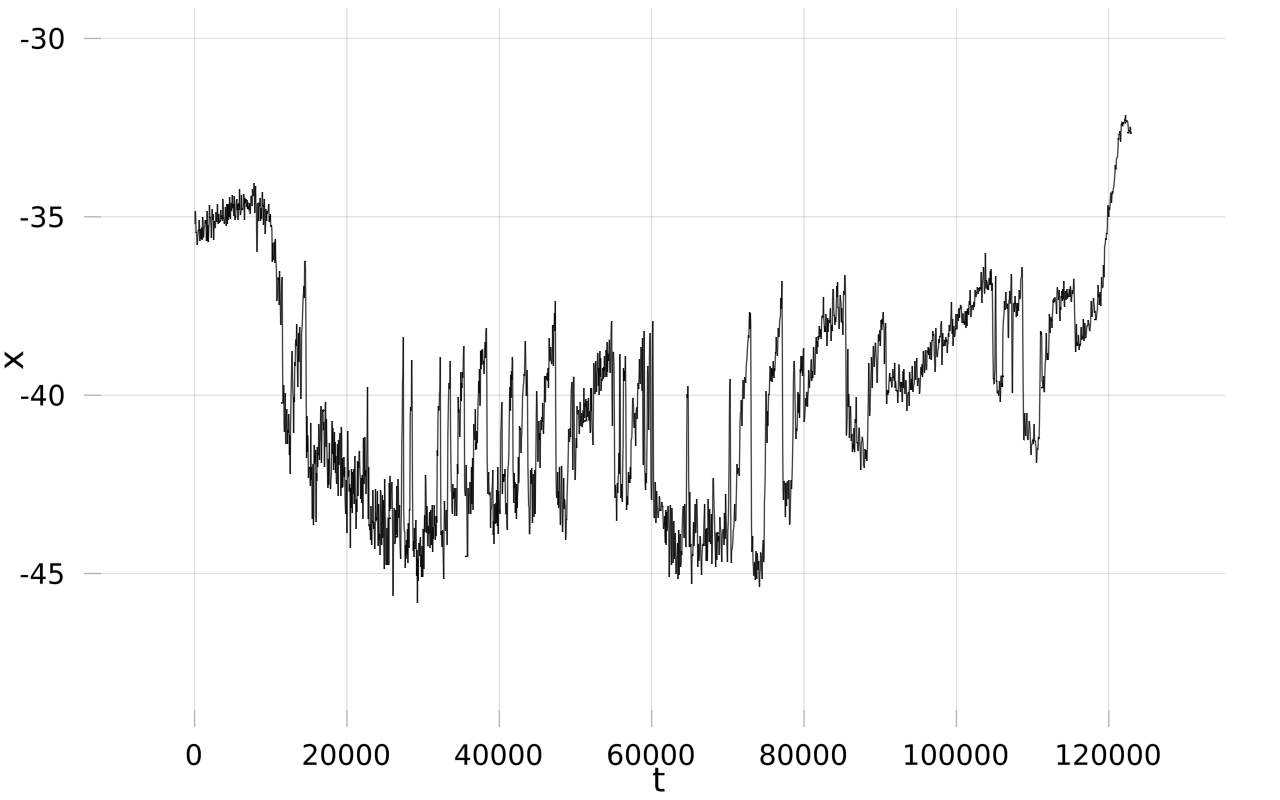}
\end{center}
\caption{NGRIP oxygen isotope $\delta_{18}\rm{O}$ measurements against time (years BP 2000). }
\label{fig:dataicecore}
\end{figure}

\begin{figure}
\begin{center}
\includegraphics[width=0.6\linewidth]{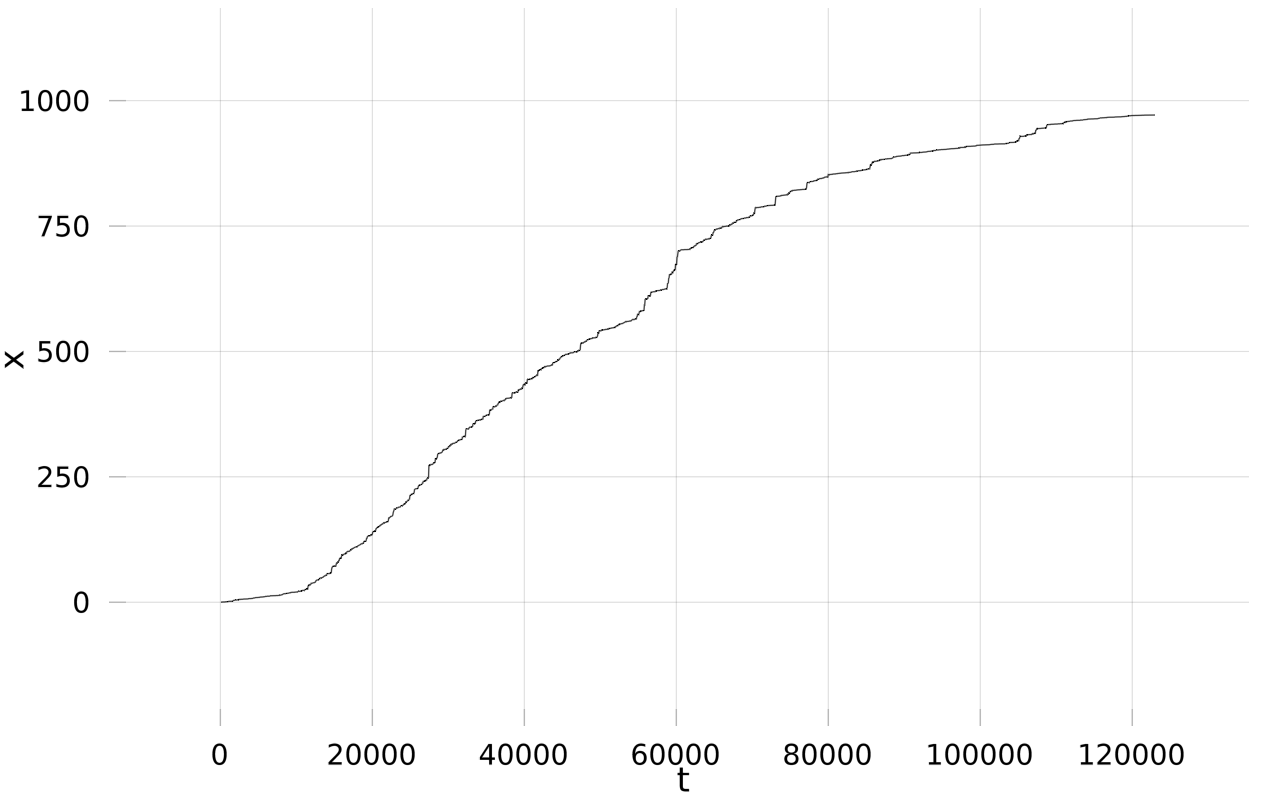}
\end{center}
\caption{Realized quadratic variation process for NGRIP oxygen isotope $\delta_{18}\rm{O}$ measurements against time (years BP 2000). }
\label{fig:obsicecore}
\end{figure}

\begin{figure}
\begin{center}
\includegraphics[width=0.6\linewidth]{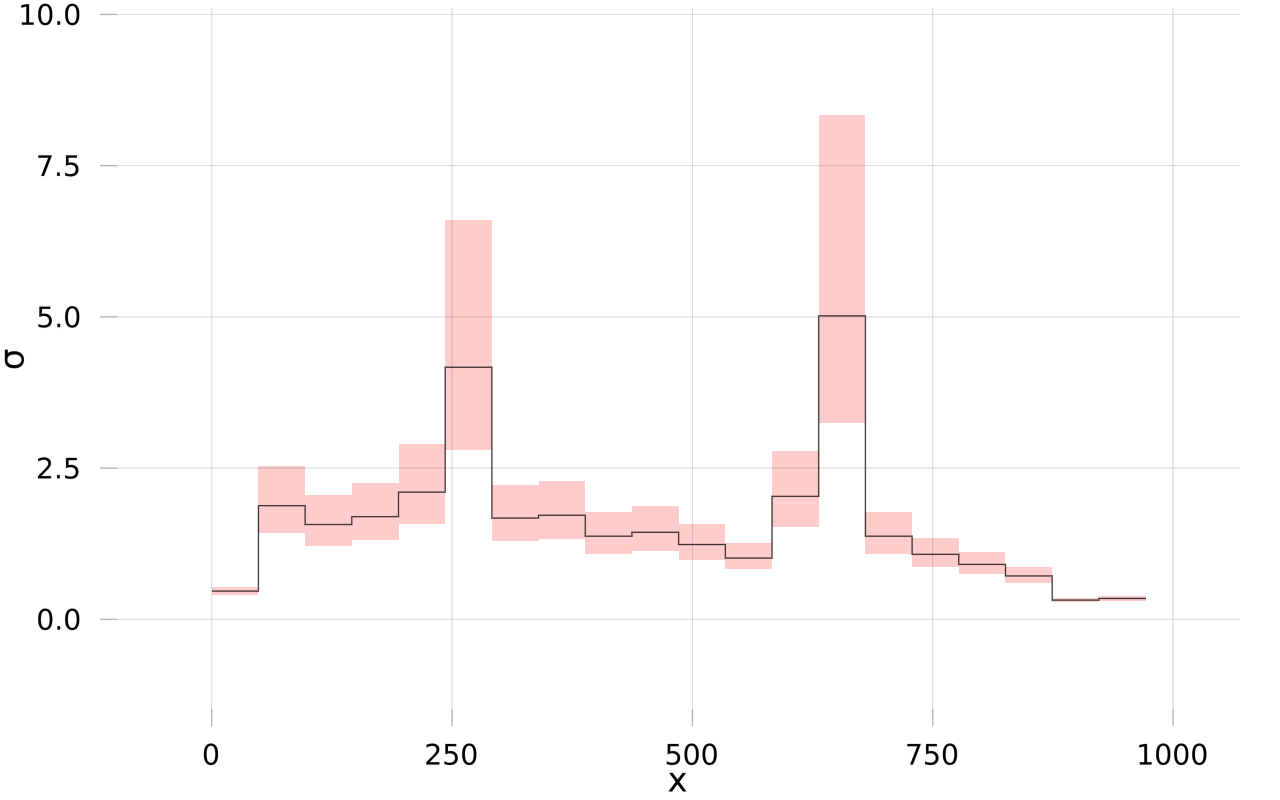}
\end{center}
\caption{Shaded: marginal $90\,\%$-posterior credible band for the piecewise constant posterior for the scale of the quadratic variation process of the ice core time series. Black: marginal posterior median.}
\label{fig:posticecore}
\end{figure}

\section{Brief discussion of some extensions}\label{sec:extensions}

In the treatment of our statistical problem  the standing assumption so far was that observations were rooted in a model derived from \eqref{eq:sde}, and that the process has been observed continuously in time. In an asymptotic setting we have derived contraction rates for the Bayesian estimator in Sections~\ref{section:pw} and \ref{section:holder}. There are (at least) two variations that immediately come to mind, an extended model that includes a drift component and a different observation regime, i.e.\ discrete time observations. We address these in the next two short sections and discuss how to adapt our previous statistical procedure to these situations.

\subsection{Inclusion of drift}

In principle, one could also consider a modification of the original SDE model by adding a drift, 
\begin{equation}\label{eq:drift}
X_t= \int_0^t a(X_s)\,\dd s+\int_0^t\sigma(X_{s-})\,\dd L_s.
\end{equation}
The property that $X$ satisfying \eqref{eq:drift} has increasing paths is guaranteed by imposing that $a$ is nonnegative.
Inclusion of the drift has the consequence that the laws of $X$ on $[0,T]$ under presence and absence of a drift, but keeping the $\int_0^t\sigma(X_{s-})\,\dd L_s$ part the same for both cases, are automatically mutually singular. This happens because an absolutely continuous change of measure in this setting only changes the third characteristic of a semimartingale, which excludes a change of drift. This implies that the drift part, $\int_0^t a(X_s)\,\dd s$, can be identified with probability one. 

Indeed, if  $X$ is observed in continuous time, one also observes the process $\sum_{s\leq \cdot}\Delta X_s$, and therefore also the difference $X- \sum_{s\leq \cdot}\Delta X_s$, which is $\int_0^{\bcdot} a(X_s)\,\dd s$. For the statistical problem of estimating $\sigma$, one can apply a modification of our procedure, by incorporating the values $a(X_s)$. For instance, if next to $\sigma$ also $a$ is piecewise constant on bins, with known (or estimated) values $a_k$, the assertion of Proposition~\ref{prop:limittau} has to be modified by taking $\Delta\bar\tau^n_k$ as $\frac{n\Delta b_k\beta}{\alpha \sigma_k+\beta a_k}$. Hence we derive analogously to Theorem~\ref{thm:contr-lip} that  for any sequence \(m_n\to \infty,\)
\[
\Pi_n\Bigl(|\xi_k-\sigma_k|>\frac{m_n}{\sqrt{n}}\Bigr)\to 0 \mbox{ in probability},
\]
for \(k=1,\ldots,K,\) where  we model  the $(\sigma_k)$ as independent random variables $(\xi_k)$ and \(\Pi_n\) stands for the posterior distribution of \((\xi_k).\)

\subsection{Discrete time observations}\label{section:discrete}

In previous sections we have assumed that $X^n$ given by~\eqref{eq:sden} is observed continuously on a long time interval $[0,T^n]$.  In this section the index $n$ plays no role  (in the situation we will describe shortly,  we  even take $n=1$  without loss of generality), it is fixed and therefore omitted below from our notation. 
The current setting is the practically relevant situation where $X$ is observed only discretely in time at timepoints $t_0, \dots, t_m$.  We use the abbreviated notation $X^d = (X_{t_i})_{i = 1}^m$. It is assumed that realisations $X_{t_i} = x_{t_i}$, $i = 0, \dots, m$ have been observed. In this discrete time setting, the level hitting times $\tau_k$,
 are only known approximately: if $b_k \in [X_{t_i},X_{t_{i+1}})$, then  $\tau_{k} \in [t_{i}, t_{i+1})$. Our previous inference procedure for piecewise constant $\sigma$, assuming  observation of the $\tau_k$,  can then still be used in a data augmentation type Gibbs sampler, targeting the joint posterior distribution of $X$ and $\sigma$ given the observations. The sampler alternates an interpolation step, where the random times $\tau_k$ are sampled conditional on $\sigma$ and the observations $X_{t_i}$, and an inference step where $\sigma$ is sampled conditional on the 
$\tau_k$  as described before to obtain the posterior distribution. In this section we propose a procedure how to sample the $\tau_k$ conditional on $\sigma$ and observations $X_{t_i}$ to complement the continuous time inference procedure developed before, e.g.\ as in Section~\ref{section:pw}. For that extension, it is enough to assume that $\sigma$ is a piecewise constant function, fixed and known. So we only consider the problem of sampling the posterior conditional on the observation of the process $X^{n}_{t_i}$ for piecewise constant $\sigma$ given by \eqref{eq:sde^n}. By the Markov property, the conditional distribution of $X$ given observations $X^d$ factorises into independent distributions of the path of $X$ on time intervals $[t_{i}, t_{i+1})$. As the conditional distribution of $\sigma$ only depends on the $\tau_k$, it is enough to sample the trajectory of $X$ on intervals $[t_{i}, t_{i+1})$ that are such that  $b_k \in [X_{t_i},X_{t_{i+1}})$ for some $k$. We therefore concentrate on the task of sampling $X$ (for this purpose we can take $n=1$ without loss of generality) conditional on $X_{S} = u$, $X_{T} = v$, $S < T$ with $b = b_k \in [u, v)$.

Simplifying, we assume that each bin contains at least one observation.
Observe that if $\sigma = 1$ then $X_{t}$,  conditional on $X_{t_i} = x_{t_i}$, $X_{t_{i+1}} = x_{t_{i+1}}$ for $t_i<t<t_{i+1}$, is just a gamma process bridge with known distribution. By Corollary~\ref{cor:lik} the law of the conditional process is absolutely continuous with respect the law of a conditional gamma process,
\begin{eqnarray*}
\frac{\dd\pp^\sigma_T}{\dd\pp^1_T}=\exp\left\{ \beta\sum_{k=1}^{K}\left[1-n\xi_{k}^{-1}\right](X_{\tau_{k}}- X_{\tau_{k-1}})-\alpha\sum_{k=1}^{K}\left(\tau_{k}-\tau_{k-1}\right)\log(\xi_{k}/n)\right\}. 
\end{eqnarray*}
We denote the law of  $\tau=( \tau_0,\ldots,\tau_K)$ conditional on $X^d = x^d$
by  $Q^\sigma_\tau$, emphasising the dependence on $\sigma$.
By the abstract Bayes formula, it holds that, with the approximation below obtained by replacing $X_{\tau_{k}}$ with $b_k$,
\[
\frac{\dd Q^\sigma_\tau}{\dd Q^1_\tau}(\tau_0, \dots, \tau_K) \approx \prod_{i = 1}^m \frac{p^\sigma(t_{i-1}, x_{t_{i-1}}; t_{i}, x_{t_{i}})}{p^1(t_{i-1}, x_{t_{i-1}}; t_{i}, x_{t_{i}})} R(\tau_0, \dots, \tau_K)
\]
where $p^\sigma(t_{i-1}, x_{t_{i-1}}; t_{i}, x_{t_{i}})$ denotes the transition density of the process $X$ with volatility $\sigma$ given $X_{t_{i-1}}=x_{t_{i-1}}$ to the value $X_{t_{i}}= x_{t_{i}}$,
and
\[
R(\tau_0, \dots, \tau_K) = \exp\left\{ \beta\sum_{k=1}^{K}\left[1-n\xi_{k}^{-1}\right](b_k - b_{k-1})-\alpha\sum_{k=1}^{K}\left(\tau_{k}-\tau_{k-1}\right)\log(\xi_{k}/n)\right\}. 
\]
A random sample of  $\tau$ under $Q^1_\tau$ can be obtained, observing that the unnormalised density of $\tau_k$ given $X_{t_{i-1}} = x_{t_{i-1}}$ and $X_{t_i} = x_{t_i}$, $b_k \in  [x_{t_i},x_{t_{i+1}})$ is \emph{approximately} given by ($\tau_k$ is now also used for a realisation of $\tau_k$)
\[
f(\tau_k) = p^1(t_{i-1}, x_{t_{i-1}}; \tau_k, b_k) p^1( \tau_k, b_k; t_i, x_{t_{i}}) 
\]
with the approximation obtained by again replacing $X_{\tau_{k}}$ by $b_k$. So the posterior distribution of $\xi_k$ given $X^d = x^d$ can be computed using a Gibbs type Metropolis-Hastings approach where in one step every $\tau_k$ is sampled conditional on $\xi_k$ given $X^d = x^d$ using  $Q^1_\tau$ as proposal distribution with $R$ as (approximate) likelihood, 
and then in a next step, $\xi_k$ are sampled given $\tau_k$ using~\eqref{eq:igposterior}.

\subsection{More general driving processes}
In this paper, we focus on SDEs with driving gamma processes. In principle, one can consider other driving L\'evy processes, but then we will lose some nice features of the current model. The situation here is analogous to the diffusion case with  Brownian noise. The Brownian noise enables one to write down a likelihood ratio `a la Girsanov. If one replaces the Brownian motion with an arbitrary continuous local martingale, the likelihood ratio will no longer be explicit. If one changes the gamma process to another increasing process in our setting, a similar phenomenon will arise. The exact expression of the likelihood ratio in Corollary 9 will change, and one has to use another prior instead of the inverse gamma one. Conceptually one can follow a similar strategy as we have proposed now, but a concrete realization will be different, and it is doubtful whether the explicit expressions for the likelihood and  the posteriori distribution exist in this case. 

\subsection*{Acknowledgments}
The research leading to these results has received funding from the European Research Council under ERC Grant Agreement 320637.
The research of the first author was supported by  the HSE University Basic Research Program and the German Science Foundation research grant (DFG Sachbeihilfe) 406700014.

\bibliographystyle{apa-good}
\bibliography{bibliography}

\begin{thebibliography}{33}
\expandafter\ifx\csname natexlab\endcsname\relax\def\natexlab#1{#1}\fi
\expandafter\ifx\csname url\endcsname\relax
  \def\url#1{{\tt #1}}\fi
\expandafter\ifx\csname urlprefix\endcsname\relax\def\urlprefix{URL }\fi

\bibitem[{Batir(2008)}]{batir2008inequalities}
Batir, N. (2008).
\newblock Inequalities for the gamma function.
\newblock {\em Archiv der Mathematik\/}, {\em 91\/}(6), 554--563.

\bibitem[{Batz et~al.(2018)Batz, Ruttor, \& Opper}]{batz17}
Batz, P., Ruttor, A., \& Opper, M. (2018).
\newblock Approximate {B}ayes learning of stochastic differential equations.
\newblock {\em Phys. Rev. E\/}, {\em 98\/}, 022109.

\bibitem[{{Belomestny} et~al.(2019){Belomestny}, {Gugushvili}, {Schauer}, \&
  {Spreij}}]{belomestny18}
{Belomestny}, D., {Gugushvili}, S., {Schauer}, M., \& {Spreij}, P. (2019).
\newblock {Nonparametric {B}ayesian inference for Gamma-type {L}\'evy
  subordinators}.
\newblock {\em Commun. Math. Sci.\/}, {\em 17\/}(3), 781--816.

\bibitem[{Belomestny et~al.(2021)Belomestny, Gugushvili, Schauer, \&
  Spreij}]{bgssweak}
Belomestny, D., Gugushvili, S., Schauer, M., \& Spreij, P. (2021).
\newblock {Weak solutions to gamma-driven stochastic differential equations}.
\newblock In preparation.

\bibitem[{Carson et~al.(2019)Carson, Crucifix, Preston, \&
  Wilkinson}]{Carson2019}
Carson, J., Crucifix, M., Preston, S.~P., \& Wilkinson, R.~D. (2019).
\newblock Quantifying age and model uncertainties in palaeoclimate data and
  dynamical climate models with a joint inferential analysis.
\newblock {\em Proceedings of the Royal Society A: Mathematical, Physical and
  Engineering Sciences\/}, {\em 475\/}(2224), 20180854.

\bibitem[{Chance et~al.(2008)Chance, Hillebrand, \& Hilliard}]{chance08}
Chance, D.~M., Hillebrand, E., \& Hilliard, J.~E. (2008).
\newblock Pricing an option on revenue from an innovation: {A}n application to
  movie box office revenue.
\newblock {\em Management Science\/}, {\em 54\/}(5), 1015--1028.

\bibitem[{Ditlevsen(1999)}]{ditlevsen1999observation}
Ditlevsen, P.~D. (1999).
\newblock Observation of $\alpha$-stable noise induced millennial climate
  changes from an ice-core record.
\newblock {\em Geophysical Research Letters\/}, {\em 26\/}(10), 1441--1444.

\bibitem[{Dufresne et~al.(1991)Dufresne, Gerber, \& Shiu}]{dufresne91}
Dufresne, F., Gerber, H.~U., \& Shiu, E. S.~W. (1991).
\newblock Risk theory with the gamma process.
\newblock {\em ASTIN Bulletin\/}, {\em 21\/}(2), 177--192.

\bibitem[{Eguchi \& Uehara(2020)}]{eguchi2020schwartz}
Eguchi, S., \& Uehara, Y. (2020).
\newblock Schwartz type model selection for ergodic stochastic differential
  equation models.
\newblock \emph{Ar{X}iv}, 1904.12398.
\newline\urlprefix\url{https://arxiv.org/abs/1904.12398}

\bibitem[{Ghosal et~al.(2000)Ghosal, Ghosh, \& Van
  Der~Vaart}]{ghosal2000convergence}
Ghosal, S., Ghosh, J.~K., \& Van Der~Vaart, A.~W. (2000).
\newblock Convergence rates of posterior distributions.
\newblock {\em Annals of Statistics\/}, (pp. 500--531).

\bibitem[{Ghosal \& van~der Vaart(2017)}]{ghosal17}
Ghosal, S., \& van~der Vaart, A. (2017).
\newblock {\em Fundamentals of nonparametric {B}ayesian inference\/}, vol.~44
  of {\em Cambridge Series in Statistical and Probabilistic Mathematics\/}.
\newblock Cambridge University Press, Cambridge.

\bibitem[{{Gugushvili} et~al.(2020){Gugushvili}, {Mariucci}, \& {van der
  Meulen}}]{gugu19}
{Gugushvili}, S., {Mariucci}, E., \& {van der Meulen}, F. (2020).
\newblock {Decompounding discrete distributions: A non-parametric Bayesian
  approach}.
\newblock {\em Scand J. Statist.\/}, {\em 47\/}(2), 464--492.

\bibitem[{{Gugushvili} et~al.(2019){Gugushvili}, {van der Meulen}, {Schauer},
  \& {Spreij}}]{gugu18b}
{Gugushvili}, S., {van der Meulen}, F., {Schauer}, M., \& {Spreij}, P. (2019).
\newblock {Nonparametric {B}ayesian volatility estimation}.
\newblock In J.~de~Gier, C.~E. Praeger, \& T.~Tao (Eds.) {\em {2017 MATRIX
  Annals}\/}, (pp. 279--302). Cham: Springer International Publishing.

\bibitem[{Gugushvili et~al.(2015)Gugushvili, van~der Meulen, \&
  Spreij}]{gugu15}
Gugushvili, S., van~der Meulen, F., \& Spreij, P. (2015).
\newblock Nonparametric {B}ayesian inference for multidimensional compound
  {P}oisson processes.
\newblock {\em Mod. Stoch. Theory Appl.\/}, {\em 2\/}(1), 1--15.

\bibitem[{Gugushvili et~al.(2018)Gugushvili, van~der Meulen, \&
  Spreij}]{gugu16}
Gugushvili, S., van~der Meulen, F., \& Spreij, P. (2018).
\newblock A non-parametric {B}ayesian approach to decompounding from high
  frequency data.
\newblock {\em Stat. Inference Stoch. Process.\/}, {\em 21\/}(1), 53--79.

\bibitem[{Gushchin et~al.(2019)Gushchin, Pavlyukevich, \&
  Ritsch}]{gushchin2019drift}
Gushchin, A., Pavlyukevich, I., \& Ritsch, M. (2019).
\newblock {Drift estimation for a L\'evy-driven Ornstein-Uhlenbeck process with
  heavy tails}.
\newblock \emph{Ar{X}iv}, 1911.11202.
\newline\urlprefix\url{https://arxiv.org/abs/1911.11202}

\bibitem[{Jasra et~al.(2019)Jasra, Kamatani, \& Masuda}]{jasra2019bayesian}
Jasra, A., Kamatani, K., \& Masuda, H. (2019).
\newblock {Bayesian inference for stable L{\'e}vy--driven stochastic
  differential equations with high-frequency data}.
\newblock {\em Scandinavian Journal of Statistics\/}, {\em 46\/}(2), 545--574.

\bibitem[{Jasra et~al.(2011)Jasra, Stephens, Doucet, \&
  Tsagaris}]{jasra2011inference}
Jasra, A., Stephens, D.~A., Doucet, A., \& Tsagaris, T. (2011).
\newblock {Inference for L{\'e}vy-driven stochastic volatility models via
  adaptive sequential Monte Carlo}.
\newblock {\em Scandinavian Journal of Statistics\/}, {\em 38\/}(1), 1--22.

\bibitem[{Kabanov et~al.(1986)Kabanov, Liptser, \&
  Shiryaev}]{kabanov1986variation}
Kabanov, Y.~M., Liptser, R.~S., \& Shiryaev, A. (1986).
\newblock On the variation distance for probability measures defined on a
  filtered space.
\newblock {\em Probability theory and related fields\/}, {\em 71\/}(1), 19--35.

\bibitem[{{Koskela} et~al.(2019){Koskela}, {Span\`o}, \& {Jenkins}}]{koskela17}
{Koskela}, J., {Span\`o}, D., \& {Jenkins}, P.~A. (2019).
\newblock {Consistency of Bayesian nonparametric inference for discretely
  observed jump diffusions}.
\newblock {\em {Bernoulli}\/}, {\em 25\/}(3), 2183--2205.

\bibitem[{Kyprianou(2014)}]{kyprianou14}
Kyprianou, A.~E. (2014).
\newblock {\em Fluctuations of {L}\'evy processes with applications.
  Introductory lectures\/}.
\newblock Universitext. Springer, Heidelberg, 2nd ed.

\bibitem[{M\"uller \& Mitra(2013)}]{mueller13}
M\"uller, P., \& Mitra, R. (2013).
\newblock Bayesian nonparametric inference---why and how.
\newblock {\em Bayesian Anal.\/}, {\em 8\/}(2), 269--302.

\bibitem[{M\"uller et~al.(2015)M\"uller, Quintana, Jara, \& Hanson}]{mueller15}
M\"uller, P., Quintana, F.~A., Jara, A., \& Hanson, T. (2015).
\newblock {\em Bayesian nonparametric data analysis\/}.
\newblock Springer Series in Statistics. Springer, Cham.

\bibitem[{Nickl \& S\"ohl(2017)}]{nickl17}
Nickl, R., \& S\"ohl, J. (2017).
\newblock Nonparametric {B}ayesian posterior contraction rates for discretely
  observed scalar diffusions.
\newblock {\em Ann. Statist.\/}, {\em 45\/}(4), 1664--1693.

\bibitem[{{North Greenland Ice Core Project
  Members}(2007)}]{northgreenlandicecoreprojectmembers2007ymoo}
{North Greenland Ice Core Project Members} (2007).
\newblock {50 year means of oxygen isotope data from ice core NGRIP}.
\newblock Supplement to: North Greenland Ice Core Project Members (2004):
  High-resolution record of Northern Hemisphere climate extending into the last
  interglacial period. Nature, 431, 147-151,
  https://doi.org/10.1038/nature02805.
\newline\urlprefix\url{https://doi.org/10.1594/PANGAEA.586886}

\bibitem[{Protter(2004)}]{Protter2004}
Protter, P.~E. (2004).
\newblock {\em Stochastic integration and differential equations\/}, vol.~21 of
  {\em Applications of Mathematics (New York)\/}.
\newblock Springer-Verlag, Berlin, 2nd ed.
\newblock Stochastic Modelling and Applied Probability.

\bibitem[{Reed \& McKelvey(2002)}]{reed02}
Reed, W.~J., \& McKelvey, K.~S. (2002).
\newblock Power-law behaviour and parametric models for the size-distribution
  of forest fires.
\newblock {\em Ecological Modelling\/}, {\em 150\/}(3), 239--254.

\bibitem[{Silverman(1986)}]{silverman86}
Silverman, B.~W. (1986).
\newblock {\em Density estimation for statistics and data analysis\/}.
\newblock Monographs on Statistics and Applied Probability. Chapman \& Hall,
  London.

\bibitem[{Todorov(2011)}]{todorov2011econometric}
Todorov, V. (2011).
\newblock Econometric analysis of jump-driven stochastic volatility models.
\newblock {\em Journal of Econometrics\/}, {\em 160\/}(1), 12--21.

\bibitem[{Tsybakov(2008)}]{tsybakov2008introduction}
Tsybakov, A.~B. (2008).
\newblock {\em Introduction to nonparametric estimation\/}.
\newblock Springer Science \& Business Media.

\bibitem[{Uehara(2019)}]{uehara2019statistical}
Uehara, Y. (2019).
\newblock {Statistical inference for misspecified ergodic L{\'e}vy driven
  stochastic differential equation models}.
\newblock {\em Stochastic Processes and their Applications\/}, {\em 129\/}(10),
  4051--4081.

\bibitem[{van Noortwijk(2009)}]{noortwijk09}
van Noortwijk, J. (2009).
\newblock A survey of the application of gamma processes in maintenance.
\newblock {\em Reliability Engineering \& System Safety\/}, {\em 94\/}(1),
  2--21.

\bibitem[{Wenocur(1989)}]{wenocour89}
Wenocur, M.~L. (1989).
\newblock A reliability model based on the gamma process and its analytic
  theory.
\newblock {\em Adv. in Appl. Probab.\/}, {\em 21\/}(4), 899--918.

\end{thebibliography}

\end{document}